\documentclass[10pt]{article}  

\usepackage{amsmath,amssymb,amsthm,amsfonts}
\usepackage {graphics}
\newtheorem{theorem}{Theorem}
\newtheorem{lemma}{Lemma}

\newtheorem{corollary}{Corollary}

\newtheorem{definition}{Definition}
\newtheorem{conjecture}{Conjecture}

\begin{document}

\begin{center}
{\bf \Large  Representation formulas for $L^\infty$ norms of weakly convergent sequences
of gradient fields in homogenization}
\end{center}
\bigskip

\begin{center}
{\bf \large Robert  Lipton and Tadele Mengesha}
\end{center}
\begin{center}
Department of Mathematics,\\
Louisiana State University,\\
Baton Rouge, LA 70803, USA.
\end{center}
\bigskip
\begin{abstract}
We examine the composition of the $L^{\infty}$ norm with weakly convergent sequences of gradient fields associated with the homogenization of second order divergence form partial differential equations  with measurable coefficients. Here the sequences of coefficients are chosen to model heterogeneous media and are  piecewise constant and highly oscillatory. We identify  local representation formulas that in the fine phase limit provide upper bounds on the limit superior of the $L^{\infty}$ norms of gradient fields. The local representation formulas are  expressed in terms of the weak limit of the gradient fields
and local corrector problems. The upper bounds may diverge according to the presence of rough interfaces. We also consider the fine phase limits for layered microstructures and for sufficiently smooth periodic microsturctures. For these cases we are able to provide explicit local formulas for the limit of the $L^\infty$ norms of the associated sequence of gradient fields. Local representation formulas for lower bounds are obtained  for fields corresponding to  continuously graded periodic microstructures
as well as for general sequences of oscillatory coefficients. The representation formulas are applied to problems of optimal material design.
\end{abstract}


\bigskip
\setcounter{theorem}{0}
\setcounter{definition}{0}
\setcounter{lemma}{0}
\setcounter{conjecture}{0}
\setcounter{corollary}{0}
\setcounter{equation}{0}
\section{Introduction}

Understanding the composition of nonlinear functionals with weakly convergent sequences is a central issue in the direct methods of the calculus of variations, homogenization theory and nonlinear partial differential equations. In this paper we discuss a composition motivated by problems of optimal design. To fix ideas consider a domain $\Omega\subset\mathbb{R}^d$, $d=2,3$, partitioned into two measurable subsets $\omega$ and  $\Omega/\omega$. Define the piecewise constant coefficient of thermal conductivity  taking the values $\alpha I$ for $x\in\omega$ and $\beta I$ for $x\in\Omega/\omega$ by $A(\omega)=(\alpha\chi_\omega+\beta(1-\chi_\omega))I$. Here $\chi_\omega$ is the characteristic function of $\omega$ with $\chi_\omega=1$ for  points in $\omega$ and zero otherwise and $I$ is the $d\times d$ identity matrix. Next consider a sequence of sets $\{\omega_n\}_{n=1}^\infty$ with indicator functions $\chi_{\scriptscriptstyle{\omega_n}}$ and the $H^1(\Omega)$ solutions $u_n$ of the boundary value problems $u_n=g$ on $\partial\Omega$ with $g\in H^{1/2}(\partial\Omega)$ and
\begin{equation}
-{\rm div}\left(A(\omega_n)\nabla u_n\right)=f
\label{divequ}
\end{equation}
for $f\in H^{-1}(\Omega)$.
The theory of homogenization \cite{Degeorgi}, \cite{Spagnolo}, \cite{Murattartar} asserts that there is a subsequence of sets, not relabeled, and a matrix valued coefficient $A^H(x)\in L^\infty(\Omega,\mathbb{R}^{d\times d})$ for which the sequence $u_{n}$ converges weakly in $H^1(\Omega)$ to $u^H\in H^1(\Omega)$ with
$u^H=g$ for $x\in\partial\Omega$ and
\begin{equation}
-{\rm div}\left(A^H\nabla u^H\right)=f.
\label{divequH}
\end{equation}
The compositions of interest are given by the $L^\infty$ norm taken over open subsets $S\subset\Omega$ and are of the form
\begin{equation}
\Vert\nabla u_{n}\Vert_{L^\infty(S)}={\rm esssup}_{x\in S}|\nabla u_{n}(x)|,
\label{morminfty}
\end{equation}
\begin{eqnarray}
\Vert\chi_{\scriptscriptstyle{\omega_n}}\nabla u_{n}\Vert_{L^\infty(S)}&\hbox{ and }&\Vert(1-\chi_{\scriptscriptstyle{\omega_n}})\nabla u_{n}\Vert_{L^\infty(S)},
\label{morminftyphase}
\end{eqnarray}
and we seek to understand the behavior of limits of the kind given by
\begin{eqnarray}\label{lims}
\liminf_{n\rightarrow\infty}\Vert\chi_{\scriptscriptstyle{\omega_n}}\nabla u_{n}\Vert_{L^\infty(S)}&\hbox{ and }&
\limsup_{n\rightarrow\infty}\Vert\chi_{\scriptscriptstyle{\omega_n}}\nabla u_{n}\Vert_{L^\infty(S)}.
\end{eqnarray}
In this paper we provide examples and  identify conditions for which it is possible to represent the limits of these compositions by  local formulas expressed in terms of the weak limit $\nabla u^H$. The representation formulas provide a multi-scale description useful for studying the composition.

To illustrate the ideas we display local formulas in the context of periodic homogenization. The unit period cell for the microstructure is $Y$ and we partition it into two sets $P$ and $Y/P$. To fix ideas we assume the set $P$ represents a single smooth particle, e.g. an ellipsoid.  The union of all particles taken over all periods is denoted by $\omega$. The coefficient $A(\omega)$ is a periodic simple function defined on $\mathbb{R}^d$ taking the value $\alpha I$ in $\omega$ and $\beta I$ in $\mathbb{R}^d/\omega$.  On rescaling by $1/n$, $n=1,2,\ldots$ the set given by the union of rescaled particles taken over all rescaled periods is denoted by $\omega_n$ and $\chi_{\omega_n}(x)=\chi_\omega(nx)$. We consider the sequence of coefficients $A(\omega_n)$ restricted to $\Omega$
and the theory of periodic homogenization \cite{BLP}, \cite{SanchezPalencia} delivers a constant  matrix $A^{H}$ of effective properties given by the formula
\begin{equation} 
A^{H}_{ij} = \int_{Y}A_{ik}(y)P_{kj}(y)dy
\end{equation}
where $P_{kj} = \partial _{x_{k}}\phi^{j}(y) + \delta_{kj}$ and $\phi ^{j}$ are  $Y$-periodic $H^{1}_{loc}(R^{d})$ solutions of the unit cell problems
\begin{equation}
\text{div}(A(y)(\nabla \phi^{j}(y) + {\bf e}^{j} ))= 0 \quad \text{in $Y$},
\end{equation}
where we have written  $A(y)=A(\omega)=(\alpha\chi_{\omega}(y)+(1-\chi_\omega(y))\beta)I$ for $y\in Y$.
It is well known that the associated energies taken over sets $S\subset\subset\Omega$ converge \cite{Spagnolo}, \cite{Murattartar}, i.e., 
\begin{eqnarray}
\lim_{n\rightarrow\infty}\int_S A^n\nabla u_n\cdot\nabla u_n dx&=&\int_S A^H\nabla u^H\cdot\nabla u^H dx\nonumber\\
&=&\int_{S\times Y} A(y)P(y)\nabla u^H(x)\cdot\nabla u^H(x)\,dydx.
\label{energies}
\end{eqnarray}
In this paper we show that the analogous formulas hold for $L^\infty$ norms and are given by the local representation formulas
\begin{eqnarray}
\lim_{n\rightarrow\infty}\Vert\chi_{\omega_n}(x)\nabla u_n\Vert_{L^\infty(S)}&=&\Vert\chi_\omega(y)P(y)\nabla u^H(x)\Vert_{L^\infty(S\times Y)},\hbox{ and }
\label{peridentinff1}\\
\lim_{n\rightarrow\infty}\Vert(1-\chi_{\omega_n}(x))\nabla u_n\Vert_{L^\infty(S)}&=&\Vert(1-\chi_\omega(y))P(y)\nabla u^H(x)\Vert_{L^\infty(S\times Y)},
\label{peridentinff2}
\end{eqnarray}
these formulas follow from Theorem \ref{periodicidentity}.

For general situations the question of finding local formulas is delicate as the
solutions of \eqref{divequ} with measurable coefficients are nominally in $H^1(\Omega)$ with gradients in $L^2(\Omega,\mathbb{R}^d)$. For sufficiently regular $f$, $g$, and $\Omega$, and in the absence of any other hypothesis on the coefficients, the theorems of Bojarski \cite{Bojarski}, for problems in $\mathbb{R}^2$, and Meyers \cite{Meyers}, for problems in $\mathbb{R}^d$, $d\geq 2$, guarantee that gradients belong to $L^{p}(\Omega,\mathbb{R}^d)$ for $2<p<p'$ with $p'$ depending on the aspect ratio $\beta/\alpha$. For the general
case one can not expect $p$ to be too large. The recent work of Faraco \cite{Faraco} shows that for $d=2$ and for $\beta=K>1$ and $\alpha=1/K$ that there exist coefficients associated with sequences of layered configurations $\omega_n$ made up of hierarchal laminations for which the sequence of gradients is bounded in $L^p_{loc}(\Omega,\mathbb{R}^d)$ for $p<p*=2K/(K-1)$ and is {\em divergent} in $L^{p}_{loc}(\Omega,\mathbb{R}^d)$ for   $p\geq p* $. This precise value for $p*$ was proposed earlier for sequences of laminated structures  using physical arguments in the work of Milton \cite{Milton}.  For measurable matrix valued coefficients $A(x)\in \mathbb{R}^{2\times 2}$ with eigenvalues in the interval $[1/K,K]$ the same critical exponent   $p*=2K/(K-1)$  holds, this result also motivated by \cite{Milton} is shown earlier in the work of Lionetti and Nesi  \cite{LionettiNesi}.

With these general results in mind we display, in section 2, a set of upper bounds on the limit superior of the compositions \eqref{morminftyphase} that hold with a minimal set of hypothesis on the sequence $\{\omega_{n}\}_{n=1}^\infty$. Here we  assume only that the sets $\omega_{n}$ are Lebesgue measurable thus the upper bound may diverge to $\infty$ for cases when these sets have corners or cusps.
The upper bound is given by a local representation formula expressed in terms of the weak limit $\nabla u^H$. It is given by the limit superior of a sequence of $L^\infty$ norms of local corrector problems driven by $\nabla u^H$.  For periodic microstructures the local correctors reduce to the well known solutions of the periodic cell problems associated with periodic homogenization \cite{BLP}, \cite{SanchezPalencia}.  In section 3 we provide a general set of sufficient conditions for which the limits \eqref{lims} agree and are given by a local representation formula see, Theorem \ref{Equality}. As before this formula is given in terms of a limit of a sequence of $L^\infty$ norms for solutions of local corrector problems driven by $\nabla u^H$. From a physical perspective the local formula measures the amplification or diminution of the gradient $\nabla u^H$ by the local microstructure. Formulas of this type have been developed earlier in the context of upper and lower bounds for the linear case \cite{liproy}, \cite{lipjmps}, \cite{lipsima} and lower bounds for the nonlinear case \cite{jimenezlipton}.

On the other hand when the boundary of the sets $\omega$ are sufficiently regular one easily constructs examples of coefficients $A(\omega)$
for which the gradients belong to $L^\infty(\Omega,\mathbb{R}^d)$. More systematic treatments developed in the  work of Bonnitier and Vogelius \cite{bonnitierVogelius}, Li and Vogelius \cite{LiVogelius}, and Li and Nirenberg \cite{LiNirenberg}  describe generic classes of coefficients $A(\omega)$ for which gradients of solutions belong to $L^\infty_{loc}(\Omega,\mathbb{R}^d)$. The earlier work  of Chipot, Kinderlehrer and Vergara-Caffarelli \cite{Chipot} establish higher regularity for coefficients $A(\omega)$ associated with laminated configurations.  
In section \ref{41} we apply the uniform convergence for simple laminates discovered in \cite{Chipot} to show that the sufficient conditions given by Theorem \ref{Equality} hold. We apply this observation to obtain an explicit local formula for the limits of compositions of the $L^{\infty}$ norm with weakly convergent sequences of gradient fields associated with layered microstructures. While in section \ref{42} we use the higher regularity theory for smooth periodic  microstructures developed  in \cite{LiNirenberg} to recover an explicit representation formula for the upper bound on the limit superior of compositions of the $L^{\infty}$ norm with weakly convergent sequences of gradient fields associated with periodic microstructures. Lower bounds on the limit inferior are developed in section 5 that agree with the upper bounds  and we recover explicit local formulas for the limits of compositions of the $L^{\infty}$ norm with weakly convergent sequences of gradient fields associated with periodic microstructures.

The $L^\infty$ norm of the field gradient inside each component material  \eqref{morminftyphase} is of interest in applications where it
is used to describe the strength of a composite structure. Here the strength of a component material is described by a threshold value of the $L^\infty$ norm of the gradient. If the $L^\infty$ norm exceeds the threshold inside $\omega_n$ then failure is initiated in that material and nonlinear phenomena such as plasticity and material degradation occur \cite{kellymac}, \cite{NuismerWhitney}.
The design of composite structures to forestall eventual failure initiation is of central interest for aerospace applications \cite{gosschristensen}.  For a given set of structural loads one seeks configurations $\omega$ that keep the local gradient field below the failure threshold inside each component material over as much of the structure as possible.  As is usual in design problems of this sort the problem is most often ill-posed (see, e.g. \cite{lipnato}) and there is no {\em best} configuration $\omega$. Instead one looks to identify sequences of configurations $\{\omega_n\}_{n=1}^\infty$  from which a {\em nearly} optimal configuration can be chosen.

The work of Duysinx and Bendsoe \cite{DesynxBendsoe} presents an insightful engineering approach to the problem of optimal design subject to constraints on the sup norm of the local stress inside a laminated material. The subsequent work of Lipton and Stuebner \cite{Liptstueb1}, \cite{Liptstueb3}, \cite{LiptstuebAIAA} develops the mathematical theory and provides numerical schemes for the design of continuously graded multi-phase elastic composites with constraints on the $L^\infty$ norm of the local stress or strain inside each material. More recent work by Carlos-Bellido, Donoso and Pedregal \cite{donsopedregal} provides the  mathematical relaxation of the $L^\infty$ gradient constrained design problem for two-phase heat conducting materials.
The feature common to all of these problems is that they involve weakly convergent sequences of gradients and their composition with $L^\infty$ norms of the type given by
\eqref{morminfty} and \eqref{morminftyphase}.
Motivated by the applications we develop an explicit local representation formula for the lower bound on \eqref{lims} for continuously graded periodic microstructures introduced for optimal design problems in \cite{liproy}, \cite{Liptstueb1}, \cite{Liptstueb3}, see section 5. A similar set of lower bounds have appeared earlier within the context of two-scale homogenization \cite{lipsima}. In section 6 we conclude the paper by outlining the connection between  optimal design problems with $L^\infty$ gradient constraints, local representation formulas, and the composition of the  $L^\infty$ norm with sequences  of gradients. Last it is pointed out that the results presented here can be extended without modification to the system of linear elasticity.

\setcounter{theorem}{0}
\setcounter{definition}{0}
\setcounter{lemma}{0}
\setcounter{conjecture}{0}
\setcounter{corollary}{0}
\setcounter{equation}{0}
\section{Mathematical background and upper bounds given by local representation formulas}

In this section we present upper bounds on the limit superior of sequences of $L^\infty$ norms of gradient fields associated with  G-convergent sequences of coefficient matrices. In what follows the coefficient matrices given by simple functions $A(x)$ taking the finite set of values $A_1,A_2,\ldots, A_N$ in the space of $d\times d$ positive definite symmetric matrices. Here no assumption on the sets $\omega_i$ where $A(x)=A_i$ are made other than that they are Lebesgue measurable subsets of $\Omega$.

We consider a sequence of coefficient matrices $A^n(x)=\sum_{i=1}^N\chi^i_n A_i$. Here $A^n(x)=A_i$ on the sets $\omega^i_n$ and the corresponding indicator function $\chi^i_n$ takes the value $\chi^i_n=1$ on $\omega^i_n$ and zero outside for $i=1,2,\dots,N$, with $\sum_{i=1}^N\chi^i_n=1$ on $\Omega$. We suppose that the sequence $\{A^n(x)\}_{n=1}^\infty$ is G-convergent with a G-limit given by the positive definite $d\times d$ coefficient matrix $A^H(x)$.
The G-limit is often referred to as the homogenized coefficient matrix.
For completeness we recall the definition of G-convergence as presented in \cite{Murattartar}:

\begin{definition}
\label{def1}
The sequence of matrices $\{A^n(x)\}_{n=1}^\infty$ is said to G-converge to $A^H(x)$ iff for every $\omega\subset\Omega$ with closure also contained in $\Omega$ and for every $f\in H^{-1}(\omega)$ the solutions $\varphi_n\in H^{1}_0(\omega)$ of
\begin{equation}
-{\rm div}\left(A^n\nabla \varphi_n\right)=f
\label{sequence}
\end{equation}
converge weakly in $H^1_0(\omega)$ to the $H^1_0(\omega)$ solution $\varphi^H$ of
\begin{equation}
-{\rm div}\left(A^H\nabla \varphi^H\right)=f.
\label{hlimit}
\end{equation}
\end{definition}

G-convergence is a form of convergence for solution operators and its relation to other notions of operator convergence are provided in \cite{Spagnolo}.
From a physical perspective each choice of right hand side $f$ in \eqref{sequence} can be thought of as an experiment
with the physical response given by the solution $\varphi_n$ of \eqref{sequence}. The physical response of heterogeneous materials  with coefficients belonging to a G-convergent sequence converge in $H_0^1(\omega)$ to that of the G-limit for every choice of sub-domain $\omega$. For sequences of oscillatory periodic and strictly stationary, ergodic random coefficients the G-convergence is described by the more well known notions of homogenization theory  \cite{BLP}, \cite{ZOK},  \cite{PapVaradan}, \cite{Spagnolo}, \cite{Murattartar}. We point out that the G-convergence described in Definition \ref{def1} is a specialization of the notion of H-convergence introduced in
\cite{Murattartar} which applies to sequences of non-symmetric coefficient matrices subject to suitable coercivity and boundedness conditions.

It is known \cite{Murattartar} that if $\{A^n\}_{n=1}^\infty$ G-converges to $A^H$, then for any $g\in H^{1/2}(\partial\Omega)$ and $f\in H^{-1}(\Omega)$,
the $H^1(\Omega)$ solutions $u_n$ of
\begin{equation}
-{\rm div}\left(A^n\nabla u_n\right)=f,\quad\quad u_{n} = g
\label{sequenceu}
\end{equation}
converge weakly in $H^1(\Omega)$ to the $H^1(\Omega)$ solution $u^H$ of
\begin{equation}
-{\rm div}\left(A^H\nabla u^H\right)=f, \quad\quad u^{H} = g.
\end{equation}

Last we recall the sequential compactness property of G-convergence \cite{Spagnolo}, \cite{Murattartar} applied to the case at hand.

\begin{theorem}
\label{compactness}
Given any sequence of simple matrix valued functions $\{A^n(x)\}_{n=1}^\infty$ there exists a subsequence $\{A^{n'}(x)\}_{n'=1}^\infty$ and a positive definite
$d\times d$ matrix valued function $A^H(x)$ such that the sequence  $\{A^{n'}(x)\}_{n'=1}^\infty$  G-converges to $A^H(x)$.
\end{theorem}

For the remainder of the paper  we will suppose that  sequence of coefficients $\{A^{n}\}_{n=1}^\infty$ G-converges to $A^H$ and we will investigate
the behavior of the gradient fields inside each of the sets $\omega^i_n$. To this end we will consider the limits
\begin{eqnarray} 
\liminf_{n\rightarrow\infty}\Vert\chi^i_n\nabla u_{n}\Vert_{L^\infty(S)} &\hbox{ and }&
\limsup_{n\rightarrow\infty}\Vert\chi^i_n\nabla u_{n}\Vert_{L^\infty(S)},\hbox{ for $i=1,2\dots,N$,}
\end{eqnarray}
where $S\subset\Omega$ is an open set of interest with closure contained inside $\Omega$.

In order to proceed we introduce the local corrector functions associated with the sequence of coefficients $\{A^{n}\}_{n=1}^\infty$. Let $Y\subset \mathbb{R}^d$ be the unit cube centered at the origin. For  $r>0$ consider $\Omega_r^{int}=\{x\in\Omega:\,dist(x,\partial\Omega)>r\}$ and
for $x\in\Omega_r^{int}$ and $z\in Y$ we introduce the $Y$ periodic  $H^1(Y)$ solution $w^{r,n}(x,z)$ of
\begin{eqnarray}
-{\rm div}_z\left(A^{n}(x+rz)(\nabla_z w^{r,n}_{\overline{e}}(x,z)+\overline{e}) \right)=0, \hbox{for $z\in Y$},
\label{correct}
\end{eqnarray}
where $\overline{E}$ is a constant vector in $\mathbb{R}^d$ with respect to the $z$ variable.
Here $x$ appears as a parameter and the differential operators with respect to the $z$ variable  are indicated by subscripts.
For future reference we note that $w^{r,n}$ depends linearly on $\overline{e}$ and we define the corrector matrix $P^{r,n}(x,z)$ to be given by
\begin{eqnarray}
P^{r,n}(x,z)\overline{e}=\nabla_z w^{r,n}_{\overline{e}}(x,z)+\overline{e}.\label{corrector}
\end{eqnarray}
We are interested in the $L^\infty$ norm associated with each phase and introduce the modulation functions ${\mathcal M}^i(\nabla u^H)$ defined for $x\in\Omega$ given by \cite{liproy}
\begin{eqnarray}
\label{FieldModulation}
{\mathcal M}^i(\nabla u^H)(x)=\limsup_{r\rightarrow 0}\limsup_{n\rightarrow\infty}\Vert\chi^i_{n}(x+rz)(P^{r,n}(x,z)\nabla u^H(x))\Vert_{L^\infty(Y)}.\label{f}
\end{eqnarray}

In what follows we will denote the measure of $\omega\subset\Omega$ by $|\omega|$ and  state the following upper bound given by a local representation formula.

\begin{theorem}
\label{upperboundd1}
Let $A^{n}$ G-converge to $A^H$ and consider any open set $S\subset\Omega$ with closure contained inside $\Omega$. There exists a subsequence, not relabeled and a sequence of decreasing measurable sets $E_n\subset S$, with $|E_n|\searrow 0$ such that

\begin{eqnarray}
\limsup_{n\rightarrow\infty}\Vert\chi^i_{n}\nabla u_n\Vert_{L^\infty(S\setminus E_n)}\leq\Vert {\mathcal M}^i(\nabla u^H)\Vert_{L^\infty(S)},\hbox{ $i=1,2,\ldots,N$.}\label{ubound1}
\end{eqnarray}
\end{theorem}

To proceed we introduce the distribution functions associated with the following sets $S^n_{i,t}$, $i=1,2,\ldots,N$, defined by
\begin{eqnarray}
S^n_{i,t}=\{x\in S:\,\chi^i_{n}|\nabla u_n|>t\}\label{seti}
\end{eqnarray}
given by
\begin{eqnarray}
\lambda_i^n(t)=|S_{i,t}^n|.\label{dist}
\end{eqnarray}

We state a second upper bound that follows from the homogenization constraint \cite{liproy}.

\begin{theorem}
\label{upperboundd2}
Let $A^{n}$ G-converge to $A^H$ and consider any open set $S\subset\Omega$ with closure contained inside $\Omega$. Suppose for $i=1,2,\ldots,N$ that
$\limsup_{n\rightarrow\infty}\Vert\chi^i_{n}\nabla u_n\Vert_{L^\infty(S)}=\ell^i<\infty$  and for every $\delta>0$ sufficiently small
there exist positive numbers $\theta^i_\delta>0$  such that
\begin{eqnarray}
\liminf_{n\rightarrow\infty}\lambda^n_{i}(\ell^i-\delta)>\theta^i_\delta.\label{nonzero}
\end{eqnarray}
 There exists a subsequence, not relabeled, such that
\begin{eqnarray}
\limsup_{n\rightarrow\infty}\Vert\chi^i_{n}\nabla u_n\Vert_{L^\infty(S)}\leq\Vert {\mathcal M}^i(\nabla u^H)\Vert_{L^\infty(S)}.\label{upbound1}
\end{eqnarray}
\end{theorem}

We provide a proof Theorem \ref{upperboundd1} noting that the proof of Theorem \ref{upperboundd2} is given in \cite{liproy}.

\begin{proof}
First note that the claim holds trivially if $\Vert {\mathcal M}^i(\nabla u^H)\Vert_{L^\infty(S)}=\infty$. Now suppose otherwise and set $\Vert {\mathcal M}^i(\nabla u^H)\Vert_{L^\infty(S)}=H<\infty$.
For this case Corollary 3.3 of \cite{liproy} shows directly that for any $\delta>0$ that the measure of the sets
\begin{eqnarray}
S^n_{i,H+\delta}=\{x\in S:\,\chi^i_n(x)|\nabla u_n(x)|>H+\delta\},
\label{si}
\end{eqnarray}
tends to zero as $n$ goes to $\infty$, i.e.,
\begin{eqnarray}
\limsup_{n\rightarrow\infty}\lambda_i^n(H+\delta)=\limsup_{n\rightarrow\infty}|S^n_{i,H+\delta}|= 0.
\label{decreasing}
\end{eqnarray}
We choose a sequence of decreasing positive numbers $\{\delta_\ell\}_{\ell=1}^\infty$, such that $\delta_\ell\searrow 0$ and from \eqref{decreasing} we can pick a subsequence of
coefficients $\{A^{n_j(\delta_1)}\}_{j=1}^\infty$ for which
\begin{eqnarray}
|S^{n_j(\delta_1)}_{i,H+\delta_1}|<2^{-j},\hbox{ $j=1,2,\ldots$}.
\label{21}
\end{eqnarray}
For $\delta_2$ we appeal again  to \eqref{decreasing} and pick out a subsequence of
$\{A^{n_j(\delta_1)}\}_{j=1}^\infty$ denoted by $\{A^{n_j(\delta_2)}\}_{j=1}^\infty$ for which
\begin{eqnarray}
|S^{n_j(\delta_2)}_{i,H+\delta_2}|<2^{-j},\hbox{ $j=1,2,\ldots$}.
\label{22}
\end{eqnarray}
We repeat this process for each $\delta_\ell$ to obtain a family of subsequences
$\{A^{n_j(\delta_\ell)}\}_{j=1}^\infty$, $\ell=1,2,\ldots$ such that $\{A^{n_j(\delta_{\ell+1})}\}_{j=1}^\infty\subset\{A^{n_j(\delta_{\ell})}\}_{j=1}^\infty$.
On choosing the diagonal sequence $\{A^{n_k(\delta_{k})}\}_{k=1}^\infty$ we form the sets
\begin{eqnarray}
E_K=\cup_{k\geq K}S_{i,H+\delta_k}^{n_k(\delta_k)}=\{x\in S:\,\chi^i_{n_k(\delta_k)}|\nabla u_{n_k(\delta_k)}|>H+\delta_k,\hbox{ for some $k\geq K$}\},
\label{diag}
\end{eqnarray}
with $E_{K+1}\subset E_K$. Noting that $|S_{i,H+\delta_k}^{n_k(\delta_k)}|<2^{-k}$, we see that $|E_K|<2^{-K+1}$.  Since $x\not\in E_K$ implies that
\begin{eqnarray}
\chi^i_{n_k(\delta_k)}|\nabla u_{n_k(\delta_k)}|<H+\delta_k\hbox{ for all } k\geq K,
\label{notinE}
\end{eqnarray}
we observe that
\begin{eqnarray}
\Vert\chi^i_{n_k(\delta_k)}\nabla u_{n_k(\delta_k)}\Vert_{L^\infty(S\setminus E_K)}<H+\delta_k\hbox{ for all } k\geq K,
\label{linfbound}
\end{eqnarray}
and we conclude that
\begin{eqnarray}
\limsup_{K\rightarrow\infty}\Vert\chi^i_{n_K(\delta_K)}\nabla u_{n_K(\delta_k)}\Vert_{L^\infty(S\setminus E_K)}\leq H,
\label{linflimbound}
\end{eqnarray}
with $|E_K|\searrow 0$ and the theorem is proved.
\end{proof}
\setcounter{theorem}{0}
\setcounter{definition}{0}
\setcounter{lemma}{0}
\setcounter{conjecture}{0}
\setcounter{corollary}{0}
\setcounter{equation}{0}
\section{Lower bounds and sufficient conditions for a local representation formula}

We suppose that  sequence of coefficients $\{A^{n})\}_{n=1}^\infty$ G-converges to $A^H$ and  investigate
the behavior of the gradient fields inside each of the sets $\omega^i_n$. Here we consider the limits
\begin{eqnarray}\label{limsi}
\liminf_{n\rightarrow\infty}\Vert\chi^i_n\nabla u_{n}\Vert_{L^\infty(S)},\hbox{ for $i=1,2\ldots,N$.}
\end{eqnarray}
and identify a general sufficient condition for obtaining a lower bound
on these quantities in terms of ${\mathcal M}^i(\nabla u^H)$.

Assume that $u_{n}$, $P^{r,n}$ and $u^{H}$ are defined as in the previous section and we consider an open
subset $S\subset\Omega$ with closure contained in $\Omega$. We write $\tau=dist(\partial S,\partial\Omega)>0$ and
set
\[
S_{\tau} = \{ x\in \Omega: dist(x,S)<\tau\}.
\]
For $r<\tau$ note that $ S\subset S_{r}\subset S_\tau\subset\Omega$. We next recall for $x\in S$
\begin{eqnarray}
{\mathcal M}^i(\nabla u^H)(x)= \limsup_{r \to 0}\limsup_{n \to \infty}\||\chi^i_n(x+ry)P^{r,n}(x,y)\nabla u^{H}(x)|^2\|_{L^{\infty}(Y)}.
\label{hyplowerlocal}
\end{eqnarray}

For this case the sufficient condition is based on the distribution function for the sequence $\{\chi^i_n(x+ry)P^{r,n}(x,y)\nabla u^{H}(x)\}$ and the lower bound is presented in the following theorem.
\begin{theorem}
\label{lowerbound1}
Let $A^{n}$ G-converge to $A^H$ and consider any open set $S\subset\Omega$ with closure contained inside $\Omega$. Suppose that
\[\|{\mathcal M}^i(\nabla u^H)\|_{L^\infty(S)}=\ell^i<\infty. \] Assume also that for all $\delta > 0$ small, there exist $\beta_{\delta}>0$ such that
\begin{equation}\label{assum}
\lim_{r\to 0}\liminf_{n \to \infty}|\{(x,y)\in S\times Y:\,|\chi_n^i(x+ry)P^{r,n}(x,y)\nabla u^{H}(x)|^2> (\ell^i)^2 - \delta \}|\geq \beta_{\delta} > 0.
\end{equation}
Then there exists a subsequence for which
\[
\lim_{r \to 0}\liminf_{n \to \infty}\|\chi^i_n\nabla u_n\|_{L^{\infty}(S_r)} \geq \|{\mathcal M}^i(\nabla u^H)\|_{L^{\infty}(S)}.
\]
\end{theorem}
\begin{proof}
Our starting point is Lemma 5.5 of \cite{casadodiazcalvogomez} which is described in the following lemma.
\begin{lemma}
\begin{eqnarray}
\lim_{r\to 0}\limsup_{n \to \infty}\int_{S}\int_{Y}|P^{r,n}(x,y)\nabla u^H(x)-\nabla u_n(x+ry)| ^{2}dydx=0.
\label{l2convergence}
\end{eqnarray}
\end{lemma}
On applying the lemma we observe that
\begin{equation}
\chi^i_n(x+ry)P^{r, n}(x,y)\nabla u^{H}(x)=\chi^i_n(x+ry)\nabla u_n(x+ry) + z^{r, n}(x,y)\quad
\forall (x,y)\in S\times Y,
\end{equation}
where
\begin{equation}\label{deco1}
\lim_{r\to 0}\limsup_{n \to \infty}\int_{S}\int_{Y}|z^{r,n}(x,y)| ^{2}dydx = 0.
\end{equation}
It follows  that 
\begin{equation}\label{deco2}
|\chi_n^i(x+ry)P^{r, n}(x,y)\nabla u^{H}(x)|^{2}=|\chi_n^i(x+ry)\nabla u_n(x+ry)|^{2} + F^{r,n}(x,y)\quad
\forall (x,y)\in S\times Y,
\end{equation}
where
\[
F^{r,n}(x,y) = |z^{r, n}(x,y)|^{2} + (z^{r, n}(x,y),\chi_n^i(x+ry)\nabla u_n(x+ry)).
\]
We show that $F^{r,n}(x,y) \to 0$ strongly in $L^{1}$ in the sense that
\[
\lim_{r\to 0}\limsup_{n \to \infty}\int_{S}\int_{
Y}| F^{r,n}(x,y)|dydx = 0.
\]
Indeed, from the definition and by Cauchy-Schwarz inequality we have
 \begin{equation}\label{small1}
\begin{split}
\int_{S}\int_{
Y}| F^{r,n}(x,y)|dydx&\leq \int_{S_{r}}\int_{Y}|z^{r, n}(x,y)|^{2}dydx \\
&+ \left(\int_{S}\int_{
Y}|z^{r, n}(x,y)|^{2}dydx\right)^{1/2}\left(\int_{S}\int_{
Y}|\nabla u^{n}(x+ry))|^{2}dydx\right)^{1/2}
\end{split}
\end{equation}
Moreover from standard a-priori estimates we know there is a constant $C>0$ independent of $r$ and $n$ for which,
\begin{equation}\label{standard1}
\int_{S}\int_{
Y}|\nabla u^{n}(x+ry))|^{2}dydx \leq C.
\end{equation}
The assertion follows on taking the limits in (\ref{small1}) and using  estimate  (\ref{standard1}) and equation (\ref{deco1}).

Now by Chebyshev's inequality, for every $\delta > 0$, we have the inequality
\[
|\{ (x,y) \in S\times Y: |F^{r,n} (x, y)| > \delta\}|\leq \frac{1}{\delta }\int_{S\times
Y}|F^{r,n}(x,y)|dydx
\]
and taking the limsup as $n\to \infty$ first and then as $r\to 0$, we see that
\begin{equation}\label{small2}
\lim_{r\to 0}\limsup_{n \to \infty}|\{ (x,y) \in S\times Y: |F^{r,n} (x, y)| > \delta\}| = 0
\end{equation}
From (\ref{deco2}) we see that
\[
\begin{split}
&\{(x,y)\in S\times Y:|\chi_n^i(x+ry)P^{r,n}(x,y)\nabla u^{H}(x)|^{2}> (\ell^i)^2 - \delta \}\subset\\
&\subset \{(x, y)\in S\times Y: |\chi_n^i(x+ry)\nabla u_n(x+ry)|^{2}> (\ell^i)^2-2\delta\}\cup\{ (x, y)\in S\times Y: |F^{r,n}(x,y)|>\delta\}.
\end{split}
\]
Therefore, applying (\ref{small2}) we obtain
\[
\begin{split}
&\lim_{r\to 0}\liminf_{n \to \infty}|\{(x,y)\in S\times Y:|\chi_n^i(x+ry)P^{r,n}(x,y)\nabla u^{H}(x)|^{2}> (\ell^i)^2 - \delta \}|\leq\\
&\leq \lim_{r\to 0}\liminf_{n \to \infty}|\{(x, y)\in S\times Y: |\chi_n^i(x+ry)\nabla u^{n}(x+ry)|^{2}> (\ell^i)^2-2\delta\}|.
\end{split}
\]
It follows from the last inequality that
\[
\lim_{r\to 0}\liminf_{n \to \infty}|\{(x,y)\in S\times Y: |\chi_n^i(x+ry)\nabla u^{n}(x+ry)|^{2} > (\ell^i)^2-2\delta\}|\geq \beta_{\delta} > 0,
\]
Here we have used our assumption ( \ref{assum}).
Therefore, there exist $R = R(\delta)$ and $N = N(\delta)$ such that
\[
|\{(x, y)\in S\times Y: |\chi_n^i(x+ry)\nabla u^{n}(x+ry)|^{2} > (\ell^i)^2-2\delta\}| >0, \quad \forall n\geq N(\delta), r\leq R(\delta).
\]
From the definition of the $ L^{\infty}$ norm it follows that,
\[
\| |\chi_n^i\nabla u^{n}|^{2}\|_{L^{\infty}(S_r)}\geq (\ell^i)^2 - 2\delta\quad \forall n\geq N(\delta), r\leq R(\delta).
\]
Taking the limit first in $n$ and then in $r$, and using the arbitrariness of of $\delta$, we get
\[
\lim_{r \to 0}\liminf_{n \to \infty}\|\chi_n^i\nabla u^{n}\|^2_{L^{\infty}(S_r)}\geq (\ell^i)^2,
\]
and the theorem follows.

\end{proof}

Last if we combine the hypotheses of theorems \ref{upperboundd2} and \ref{lowerbound1} we obtain a sufficient condition for a local representation formula for limits of compositions of the $L^\infty$ norm with
weakly convergent sequences of gradients associated with homogenization.

\begin{theorem}
\label{Equality}
Let $A^{n}$ G-converge to $A^H$ and consider any open set $S\subset\Omega$ with closure contained inside $\Omega$. Suppose for sufficiently small $r<\tau$, $ S\subset S_{2r}\subset S_\tau\subset\Omega$, for $i=1,2,\ldots,N$ that
$\limsup_{r \to 0}\limsup_{n\rightarrow\infty}\Vert\chi^i_{n}\nabla u_n\Vert_{L^\infty(S_r)}=\ell^i<\infty$  and for every $\delta>0$ sufficiently small
there exist positive numbers $\theta^i_\delta>0$  such that
\begin{eqnarray}
\label{nonzero2-1}
\limsup_{r \to 0}\limsup_{n\rightarrow\infty}|\{x\in S_r:\,\chi_n^i|\nabla u_n|>\ell^i-\delta)\}|\geq\theta^i_\delta>0,
\end{eqnarray}
in addition suppose that $\limsup_{r \to 0}\|{\mathcal M}^i(\nabla u^H)\|_{L^\infty(S_r)}=\tilde{\ell}^i<\infty$ and for all $\delta > 0$ small, there exist $\beta_{\delta}>0$ such that
\begin{equation}\label{assum2}
\lim_{r\to 0}\liminf_{n \to \infty}|\{(x,y)\in S_r\times Y:\,|\chi_n^i(x+ry)P^{r,n}(x,y)\nabla u^{H}(x)|^2> (\tilde{\ell}^i)^2 - \delta \}|\geq \beta_{\delta} > 0.
\end{equation}
There exists a subsequence, not relabeled, such that
\begin{eqnarray}
\lim_{r \to 0}\lim_{n\rightarrow\infty}\Vert\chi^i_{n}\nabla u_n\Vert_{L^\infty(S_r)}=\lim_{r \to 0}\Vert {\mathcal M}^i(\nabla u^H)\Vert_{L^\infty(S_r)}.\label{equal1}
\end{eqnarray}
\end{theorem}

\setcounter{theorem}{0}
\setcounter{definition}{0}
\setcounter{lemma}{0}
\setcounter{conjecture}{0}
\setcounter{corollary}{0}
\setcounter{equation}{0}
\section{Local representation formula for layered and periodic microstructures}
We now describe sequences of configurations  for which one has equality in the spirit of \eqref{equal1}. The first class of configurations are given by  sequences of finely layered media. The second class is given by a sequence of progressively finer periodic microstructures comprised of inclusions with smooth boundaries.  In what follows the results of \cite{Chipot} provide the sufficient conditions \eqref{nonzero2-1} and \eqref{assum2} for  the case of finely layered media. While the higher regularity results of \cite{LiVogelius} and \cite{LiNirenberg} allow for the computation of an upper bound for the periodic case. This upper bound agrees with an explicit lower bound developed in section 5. We note that the lower bound for the periodic case can also be obtained using the earlier results given in \cite{lipsima}.

In order to proceed let us recall the fundamental results from homogenization theory for periodic media.
We denote a $d$ dimensional cube centered at $x$ and of side length $r$ by $Y(x,r)$. For the unit cube centered at the origin we abbreviate the notation and write $Y$. The coefficient $A(y)$ is a periodic simple function defined on the unit period cell $Y$ taking the $N$ values $A_i$, $i=1,\ldots,N$ in the space of positive symmetric $d\times d$ matrices. We denote the indicator functions of the sets $Y_i$ where $A(y)=A_i$ by $\chi^i$ and write $A(y)=\sum_{i=1}^N A_i\chi^i(y)$.
It is well known from the theory of periodic homogenization \cite{BLP} that the sequence of coefficients $A^{n}(x) = A(nx)$ $G-$ converge to the homogenized constant matrix $A^{H}$ given by the formula
\begin{equation}\label{effectivematrix}
A^{H}_{ij} = \int_{Y}A_{ik}(y)P_{kj}(y)dy
\end{equation}
where $P_{kj} = \partial _{x_{k}}\phi^{j}(y) + \delta_{kj}$ and $\phi ^{j}$ are  $Y$-periodic $H^{1}_{loc}(R^{d})$ solutions of the cell problems
\begin{equation}\label{cellproblemperiodic}
\text{div}(A(y)(\nabla \phi^{j}(y) + {\bf e}^{j} ))= 0 \quad \text{in $\mathbb{R}^{d}$},
\end{equation}
where this equation is understood in the weak sense, i.e.,  
\begin{equation}
\label{periodicweakform}
\int_{Y} (A(y)(\nabla \phi^{j}(y) + {\bf e}^{j} ), \nabla \psi)dy = 0,  \quad \forall \psi \in H^{1}_{per}(Y).
\end{equation}
For periodic microstructures we define the  modulation function by
\begin{eqnarray}
{\mathcal M}^i(\nabla u^H)(x)=\||\chi^i(\cdot)P(\cdot)\nabla u^{H}(x)| \|_{L^\infty(Y)} \quad i=1,\cdots,N.
\end{eqnarray}

\subsection{Laminated microstructure}
\label{41}
The layered configurations as introduced in this section are a special class of periodic configurations.
To fix ideas we consider a two dimensional problem and partition the unit period square  $Y\subset \mathbb{R}^{2}$ for the layered material as follows:
\[
Y_{1}  = \{(y_{1},y_{2})\in Y: -\frac{1}{2}\leq y_{1}\leq-\frac{1}{2}+\theta\quad Y_{2} = \{(y_{1},y_{2})\in Y: -\frac{1}{2}+\theta\leq y_{1}\leq \frac{1}{2}\}
\]
where $\theta$ is a specified value in the interval $(0, 1).$ Let $\chi^1$ and $\chi^2$ denote the indicator functions of $Y_1$ and $Y_2$ respectively and
consider  
the Y-periodic matrix function  $A(y)$  given by
\[
A(y) = \alpha I \chi^1(y) + \beta I \chi^2(y),\]
for positive constants $\alpha<\beta$. $I$ is the $2\times 2$ identity matrix.
Let $\Omega\subset\mathbb{R}^2$ and  $u_{n}$ be the  $H^{1}(\Omega)$ solution to
\begin{equation}
-{\rm div}\left(A(nx)\nabla u_n\right)=f  ~~ \text{in $\Omega$ and} \quad u^{n}=0~~ \text{on $\partial \Omega$}.
\end{equation}
Then  $ u_{n}$ converges weakly in $H^1(\Omega)$ as $n\rightarrow\infty$ to the $H^1(\Omega)$ solution $u^H$ of
\begin{equation}
-{\rm div}\left(A^H\nabla u^H\right)=f, ~~ \text{in $\Omega$ and} \quad u^{H}=0~~ \text{on $\partial \Omega$}.
\label{hlimitu}
\end{equation}
where  $A^{H}$ is determined using the formula \eqref{effectivematrix}.
The gradient of solutions of the cell problem \eqref{cellproblemperiodic} for layered materials are given by
\[
\begin{split}
\nabla \phi^{1}(y) & = \left(\frac{(1-\theta)(\beta -\alpha)}{\theta\beta + (1-\theta)\alpha}\chi^1(y) +  \frac{\theta(\beta - \alpha)}{\theta\beta + (1-\theta)\alpha}\chi^2(y)\right){\bf e}^{1}
\end{split}
\]
and
\[
\nabla \phi^{2}(y) = {\bf e }^{2}\quad \text{for all $y\in Y$}.
\]

We define the constants
\begin{eqnarray}
a_{h}=\frac{\alpha \beta}{\theta\beta + (1-\theta)\alpha}~~\text{ and}~~ a_{m}=\theta\alpha+(1-\theta)\beta,
\label{arithhar}
\end{eqnarray}
and introduce the $Y$  periodic scalar coefficient $a(y)=\alpha\chi^1(y)+\beta\chi^2(y)$.
A simple calculation gives
\[
P(y)= \begin{bmatrix}
     p_{11}(y)    &0\\
0&1
        \end{bmatrix}\quad\text{where }~~ p_{11}(y)=\frac{a_h}{a(y)}
\]
The homogenized matrix $A^{H}$ is given by
\[
A^{H} = \begin{bmatrix}
         a_{h}&0\\
0&a_{m}
        \end{bmatrix}.
\]
The modulation function for each phase is given by:
\[
\begin{split}
{\mathcal M}^{1}(\nabla u^H)(x) &= \sqrt{\left(\frac{\beta}{\theta\beta + (1-\theta)\alpha}\partial_{x_{1}} u^{H}\right)^{2} + (\partial_{x_{2}}u^{H})^{2}}\\
{\mathcal M}^{2}(\nabla u^H)(x) &= \sqrt{\left(\frac{\alpha}{\theta\beta + (1-\theta)\alpha}\partial_{x_{1}} u^{H}\right)^{2} + (\partial_{x_{2}}u^{H})^{2}}
\end{split}
\]
We  now apply the  regularity and convergence  results associated with G-convergent coefficients for sequences of layered materials \cite{Chipot}.   For right hand sides $f\in H^{1}(\Omega)$ there exists a $p>2$ such that for any subdomain $\Omega'\Subset \Omega$
\[u^{n} \in H^{1, \infty}(\Omega')\quad \text{and}~~ \partial _{x_{2}}u_n, ~a(nx)\partial_{x_{1}}u_n\in H^{1, p}(\Omega')\]
 with the estimate that for some  $C=C(\alpha, \beta,\Omega', \Omega)$,
\begin{equation}
\label{Chipotestimate}
\|\partial _{x_2}u_n\|_{H^{1,p}(\Omega')} + \|a(nx)\partial_{x_1}u_n\|_{H^{1, p}(\Omega')} \leq C\|f\|_{H^{1}(\Omega)},   
\end{equation}
see \cite{Chipot}.
The Sobolev embedding theorem implies that $\{\partial _{x_{2}}u_n\}_{n=1}^\infty $ and $ \{a(nx)\partial_{x_{1}}u_n\}_{n=1}^\infty$ are equicontinuous families over $\Omega'$ and  uniformly bounded in $C(\Omega')$. Then from \eqref{Chipotestimate} and the weak convergence $u_n \rightharpoonup u^{H}$ in $H^{1}(\Omega)$ it follows that  for a subsequence
\begin{equation}
\partial _{x_{2}}u_n \to \partial _{x_{2}} u^{H}, \quad \quad a(nx)\partial _{x_{1}}u_n \to a_{h}\partial _{x_{1}} u^{H}\quad \text{uniformly in $\Omega'$}.
\label{uni}
\end{equation}
We observe that
\begin{equation}
\alpha|\partial_{x_1}u_n-p_{11}(nx)\partial_{x_1}u^H|\leq a(nx)|\partial_{x_1}u_n-p_{11}(nx)\partial_{x_1}u^H|=|a(nx)\partial_{x_1}u_n-{a_h}\partial_{x_1}u^H|,
\label{identuniform}
\end{equation}
and on applying \eqref{uni} and noting that $P(y)$ is constant inside each phase we see for $i=1,2$ that
\begin{eqnarray}
 |\chi^{i}(nx)\nabla u_n| &=& |\chi^{i}(nx)P(nx)\nabla u^{H}| +  m^{i}_{n}(x)\\
\label{uniformremainder}
&=&{\mathcal M}^{i}(\nabla u^{H})(x)+ m^{i}_{n}(x)
\label{modanduniform}
\end{eqnarray}
where $m_{n}^{i}(x) \to 0$ uniformly in $\Omega'$. Hence we arrive at the local representation formula for layered microstuctures
given by
\begin{theorem}
\label{layerlocal}
\begin{eqnarray}
\label{laminateconverge}
\lim_{n\to \infty}\|\chi^{i}(nx)\nabla u_n \|_{L^{\infty}(\Omega')} &=
& \| {\mathcal M}^{i}(\nabla u^{H})\|_{L^{\infty} (\Omega')}.
\end{eqnarray}
\end{theorem}
It is easily seen that the  uniform convergence implies that sequence of the gradients$\{\nabla u_n\}$ satisfy the  non-concentrating conditions given by \eqref{nonzero2-1}. Indeed, setting
\\$L^{i} = \lim_{n\to \infty}\|\chi^{i}(nx)\nabla u_n\|_{L^{\infty}(\Omega')},$  and for any $\delta>0$ there exists sufficiently large $n$ for which $|m^{i}_{n}(x)|<\frac{\delta}{2}$ for $x\in\Omega'$ and
\[|\chi^i(nx)\nabla u_n(x)|>|{\mathcal M}^i\nabla u^H(x)|-\frac{\delta}{2}\]
so
\begin{eqnarray}
\{ x\in \Omega': {\mathcal M}^{i}(\nabla u^{H}(x))> L^{i} -\frac{\delta}{2}  \}&\subset
& \{x\in \Omega': |\chi^{i}(nx)\nabla u_n|> L^{i} -\delta \}.\nonumber
\end{eqnarray}
Therefore we conclude that for $L^i>\delta > 0$
\[
\liminf_{n\to \infty}|\{x\in \Omega':  |\chi_{i}(nx)\nabla u_n|> L^{i} -\frac{\delta}{2} \} |\geq |\{ x\in \Omega':{\mathcal M}^{i}(\nabla u^{H})(x) > L^{i} -\delta  \} | > 0.
\]
Last the non-concentrating condition \eqref{assum2} follows immediately from the piecewise constant nature of the corrector matrix $P(y)$ for layered materials.


\subsection{Periodic microstructure}
\label{42}
We consider  periodic microstructures associated with particle and fiber reinforced composites. As before we divide  $Y$ into a union of $N$ disjoint subdomains $Y_{1}\dots Y_{N}$. Instead of  proceeding within the general context developed in  \cite{LiNirenberg}, \cite{LiVogelius} we fix ideas we suppose that the domains
$Y_1,\dots,Y_{N-1}$ denote convex particles with smooth (i.e., $C^2$) boundaries embedded inside a connected phase described by the domain $Y_N$, see Figure \ref{Particles}. As before we denote the indicator function of $Y_i$ by $\chi^i$ and the $Y$ periodic coefficient is written $A(y)=\sum_{i=1}^N\chi^i(y)A_i$ with
each $A_i$ being a symmetric $d\times d$ matrix of constants satisfying the coercivity and boundedness conditions given by
\[
\lambda |\xi|^{2}\leq (A_i\xi, \xi)\leq \Lambda |\xi|^{2}\quad \forall \xi \in \mathbb{R}^{d},\hbox{ and $i=1,\ldots,N$}.
\]
For any bounded domain $\Omega \subset R^{d}$ we consider the $H^{1} (\Omega)$ solutions $u_{n}$ of
\begin{equation}\label{basicpde}
\text{div}(A(nx)\nabla u_{n} )= 0 \quad \text{in $\Omega$}
\end{equation}
associated with prescribed Neumann or Dirichlet boundary conditions. From the theory of periodic homogenization the solutions converge weakly in $H^{1}$ to the homogenized solution $u^{H}$.
In this section we establish the following local representation theorem.
\begin{theorem}\label{periodicidentity}
Let $A(y)$ and the subdomains $\{Y_{i}\}_{i=1}^N$ be as described above. Suppose $u_{n}$ solves \eqref{basicpde} and $u^{H}$ is the corresponding homogenized solution, then  for any subdomain $\Omega'$ compactly contained inside $\Omega$ one has the local representation formula given by
\begin{equation}
\lim_{n\to \infty} \|\chi^i(nx)\nabla u_{n}\|_{L^{\infty} (\Omega '))} = \|{\mathcal M}^i(\nabla u^H)\|_{L^{\infty}(\Omega ')}.
\label{periodicidentityy}
\end{equation}
\end{theorem}
For the proof we will use the $W^{1,\infty}$ estimate for weak solutions of linear equations with oscillatory periodic coefficients  obtained in \cite{AvellanedaLin} for smooth coefficients and later extended in \cite{LiNirenberg} to include discontinuous but locally H\"older coefficients. A $W^{1,p}$ estimate for $p<\infty$ is given in \cite{Caffarelliperal}. We point out that we have restricted the discussion to periodic homogenization for  particle reinforced configurations of the kind illustrated in Figure \ref{Particles}.  However the regularity theory for oscillatory periodic coefficients developed in  \cite{LiNirenberg} applies to more general types of domains $Y_1,\ldots,Y_N$ with $C^{1,\alpha}$ boundaries. We note that the proof given here goes through verbatim for period cells with coefficients satisfying the general hypotheses described in  \cite{LiNirenberg}.

\begin{figure}[tbp]
\centerline{\scalebox{0.3}{\includegraphics{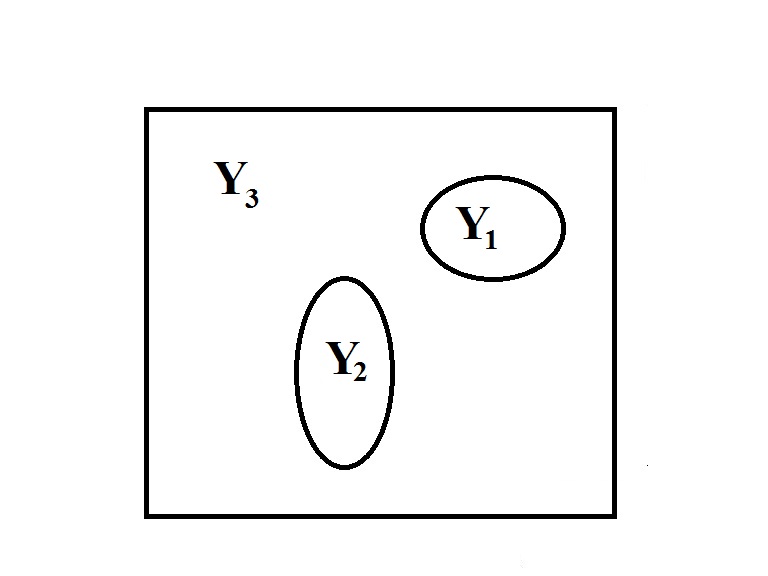}}}
\caption{Particle reinforced geometry for two inclusions $Y_1$ and $Y_2$.}
\label{Particles}
\end{figure}
Theorem 1.9 of \cite{LiNirenberg} and a suitable rescaling shows that for any $r>0$ and $ Y(x_{0}, r)\subset \Omega$ that there exists a positive constant $C$ independent of $r$ and $n$ for which
\begin{equation}
\|\nabla u_{n}\|_{L^{\infty}(Y(x_{0},r/2))}\leq Cr^{-1}\|u_{n}\|_{L^{\infty}(Y(x_{0}, r))}.
\label{linftyest}
\end{equation}
The local $L^{\infty}$ estimate for weak solutions of elliptic linear problems (Theorem 8.17, \cite{Gilbarg-Trudinger}) gives
\begin{equation}
\|u_{n}\|_{L^{\infty}(Y(x_{0}, r))} \leq C r^{-d/2} \|u_{n}\|_{L^{2}(Y(x_{0}, 2r))},
\label{linfltwo}
\end{equation}
where the constant $C$ is independent of $n$ and $r$.
Combining the  two estimates delivers the following lemma.
\begin{lemma}\label{w1inftyestimate}
Let $A(y)$ and the subdomains $\{Y_{i}\}$ be as described above. Choose $r\in (0, 1)$ such that $Y(x_{0}, 2r)\subset \Omega$.
Then if  $u_{n}$ solves \eqref{basicpde}, then there exists $C$, independent of $n$ and $r$ such that
\begin{equation}
\|\nabla u_{n}\|_{L^{\infty}(Y(x_{0},r/2))}\leq Cr^{\frac{-(d+2)}{2}}\|u_{n}\|_{L^{2}(Y(x_{0}, 2r))}.
\label{lemma2}
\end{equation}
\end{lemma}
\begin{proof}[Proof of Theorem \ref{periodicidentity}]
To prove the theorem we first show that there is a subsequence, $n_k$, for which
\begin{eqnarray}
\lim_{k\rightarrow \infty}\Vert\nabla u_{n_k}(x)-P(n_k x)\nabla u^H(x)\Vert_{L^\infty(\Omega')}.
\label{seqrep}
\end{eqnarray}
To begin we choose
$x_{0}\in \Omega$ and  $r>0$ such that $rn$ is an integer and $Y(x_{0}, r) \subset \Omega$ contains an integral number of periods of diameter $1/n$. Then from \eqref{cellproblemperiodic} we see that $ (1/n) \phi^{j}(n)$ is a $Y(x_{0}, r)$-periodic $H^{1}_{loc}$ function satisfying
\begin{equation}\label{basiccellpde}
\text{div}(A(nx)(\nabla (\frac{1}{n} \phi^{j}(nx))  + e^{j}))= 0 \quad \text{in $\mathbb{R}^d$}
\end{equation}
Combining equations (\ref{basicpde}) and (\ref{basiccellpde}) we note that
\begin{equation}\label{combinedpde}
\text{div}(A(nx)[\nabla u_{n}-(\nabla w_{n}(x, x_{0}) + \nabla u^{H}(x_{0}))] )= 0 \quad \text{in $Y(x_{0}, r)$}
\end{equation}
where
\[
w_{n}(x,x_{0}) = \sum^{d}_{j} \frac{1}{n} \phi^{j}(nx))  \partial_{x_{j}} u^{H}(x_{0}) +  u^{H}(x_{0}).
\]
Observe that for this choice of $Y(x_0,r)$
\[
\nabla w_{n}(x, x_{0}) + \nabla u^{H}(x_{0}) =P(n x) \nabla u^{H}(x_{0})
\]
Adding and subtracting $P(n x) \nabla u^{H}(x_0) $ delivers
\begin{eqnarray}\label{beginestimate}
&&\|\nabla u_{n}(x)-P(nx) \nabla u^{H}(x)\|_{L^{\infty} (Y(x_{0}, r/2))) }
\leq \|\nabla u_{n}(x)-P(nx) \nabla u^{H}(x_0)\|_{L^{\infty} (Y(x_{0}, r/2))) }\nonumber\\
&&+ \|P(nx) \nabla u^{H}(x)-  P(nx) \nabla u^{H}(x_{0})\|_{L^{2}(Y(x_{0}, r/2)))}.
\end{eqnarray}
We apply Lemma \ref{w1inftyestimate} to find  a constant $C$ independent of $n$ and $r$ such that the following estimate holds true:
\[
\begin{split}
\|\nabla u_{n}-&P(nx) \nabla u^{H}(x_{0}) \|_{L^{\infty} (Y(x_{0}, r/2))) }\\
&\leq \frac{C}{r^{-(d+2)/2}}\|u_{n} - (w_{n} + \nabla u^{H}(x_{0}) \cdot(x-x_{0}))\|_{L^{2}(Y(x_{0}, r)))}
\end{split}
\]
Combining with \eqref{beginestimate} we obtain
\begin{equation}\label{basicestimate}
\begin{split}
\|\nabla u_{n}-&P(nx) \nabla u^{H}(x)\|_{L^{\infty} (Y(x_{0}, r/2))) }\\
&\leq \frac{C}{r^{-(d+2)/2}}\|u_{n} - (w_{n} + \nabla u^{H}(x_{0}) \cdot(x-x_{0})) \|_{L^{2}(Y(x_{0}, r)))} \\
&+ \|P(nx) \nabla u^{H}(x)-  P(nx) \nabla u^{H}(x_{0})\|_{L^{2}(Y(x_{0}, r)))}.
\end{split}
\end{equation}
We bound the the first and second terms on the righthand side of \eqref{basicestimate}.
The  first term in the right hand side of \eqref{basicestimate} is bounded above by
\begin{equation}\label{basicestimate2}
\begin{split}
\|u_{n} - (w_{n} + \nabla &u^{H}(x_{0}) \cdot(x-x_{0}) )\|_{L^{2}(Y(x_{0}, r)))}\\
 &\leq \| u^{H}(x) - (u^{H} (x_{0}) + \nabla u^{H}(x_{0}) \cdot(x-x_{0}))\|_{L^{2}(Y(x_{0}, r))}\\
& +\| u_{n} - u^{H}\|_{L^{2} (Y(x_{0}, r))} + \|  \sum^{d}_{j} \frac{1}{n}\phi^{j}(n x)  \partial_{x_{j}} u^{H}(x_{0})\|_{L^{2}(Y(x_{0}, r))}\\
\end{split}
\end{equation}
 We apply Lemma \ref{w1inftyestimate} together with a priori elliptic estimates to find that
\begin{equation}\label{linftyboundofcorrector}
\|\nabla \phi^j(n\cdot) \|_{L^{\infty}(\mathbb{R}^{d})} \leq C,
\end{equation}
where $C$ is independent of $n$.
Moreover, as $u^H$ is a solution of a PDE in divergence form with constant coefficients it satisfies
\begin{equation}
\label{smoothnessofuh}
\begin{split}
|u^{H}(x) - (u^{H} (x_{0}) + \nabla u^{H}(x_{0}) \cdot(x-x_{0})) |&\leq M|x-x_{0}|^{2},\hbox{ $x\in\Omega'$}\\
| \nabla u^{H}(x) - \nabla u^{H}(x_{0})| &\leq M |x-x_{0}|, \hbox{ $x\in\Omega'$}
\end{split}
\end{equation}
where $M$ is the supremum of $|D^{2} u^{H}(x)|$ over $\Omega'$.
Applying \eqref{linftyboundofcorrector} and \eqref{smoothnessofuh} gives
\begin{equation}\label{basicestimate3}
\begin{split}
\|u_{n} - (w_{n} + \nabla &u^{H}(x_{0}) \cdot(x-x_{0}) )\|_{L^{2}(Y(x_{0}, r)))}\\
&\leq C\left(r^{2+d/2}M + \frac{1}{\sqrt{n}}+\| u_{n} - u^{H}\|_{L^{2} (Y(x_{0}, r))}\right).
\end{split}
\end{equation}
for some constant $C$ independent of $r$ and $n$.
From periodicity it follows  that
$$\Vert \nabla\phi^j(n x)\Vert_{L^\infty(Y(x_0,r))}=\Vert \nabla\phi^j(y)\Vert_{L^\infty(Y)}$$
and applying Lemma \ref{w1inftyestimate}  together with  \eqref{smoothnessofuh} delivers
\begin{equation}
\label{linftyboundofcorrector2}
\|P(nx) \nabla u^{H}(x)-  P(nx) \nabla u^{H}(x_{0})\|_{L^{2}(Y(x_{0}, r))}\leq C r^{(d+2)/2}
\end{equation}
From the theory of periodic homogenization see, \cite{BLP}, \cite{ZOK}, one has the convergence rate given by
\[
\| u_{n} - u^{H}\|_{L^{2} (Y(x_{0}, r))}\leq  C\frac{1}{\sqrt{n}}
\]
and collecting results we have
\begin{eqnarray}
\|\nabla u_{n}- &P(nx) \nabla u^{H}(x) \|_{L^{\infty} (Y(x_{0}, r/2))) }\leq C\left(Mr + \frac{1}{\sqrt{n r^{d+2}}}\right).
\label{rnk}
\end{eqnarray}
We pass to a subsequence $n_k$, and consider $Y(x_0,r_k)$  such that, $r_{k}\to 0$, $r_{k}n_{k}$ is an integer, and $r_{k}^{(d+2)/2}n_{k}^{1/2} \to \infty$ as $k\to \infty.$
Then
consider any subdomain $\Omega '\subset\subset \Omega $,  and cover it with cubes $\{ Y(x_{i}, r_{k}/2)\}_{x_{i}\in\Omega '  }$. Using compactness we  choose finitely many cubes so that
\[
\Omega '\subset \cup_{i=1}^{L} Y(x_{i}, r_{k}/2),
\]
Now since $\|\nabla u_{n_{k}}-P(n_{k}x) \nabla u^{H}(x) \|_{L^{\infty} (\Omega '))} $ is bounded above by the $L^{\infty}$ norms over a finite collection of cubes, we see that
\[
\|\nabla u_{n_{k}}-P(n_{k}x) \nabla u^{H}(x) \|_{L^{\infty} (\Omega '))}\leq  C\left(r_{k} + \frac{1}{\sqrt{n_k r_k^{d+2}}}\right)
\]
for sufficiently large $k$ to conclude
\begin{eqnarray}
\lim_{k\to \infty}\|\nabla u_{n_{k}}\|_{L^{\infty} (\Omega '))} = \lim_{k \to \infty }\|P(n_{k}x) \nabla u^{H}(x) \|_{L^{\infty} (\Omega '))},
\label{cuniformconvg}
\end{eqnarray}
so
\begin{eqnarray}
\lim_{k\to \infty}\|\chi^i(n_k x)\nabla u_{n_{k}}\|_{L^{\infty} (\Omega '))} = \lim_{k \to \infty }\|\chi^i(n_k x)P(n_{k}x) \nabla u^{H}(x) \|_{L^{\infty} (\Omega '))}.
\label{cuniformconvgi}
\end{eqnarray}
%
%
Now we bound  \eqref{cuniformconvgi} from above and below by $\Vert{\mathcal M}^i(\nabla u^H)\Vert_{L^\infty(\Omega')}$.
First note for each $n_k$ and $x\in\Omega'$ that
\begin{eqnarray}
\label{upbddd}
|\chi^i(n_k x)P(n_{k}x) \nabla u^{H}(x)|\leq\Vert\chi^i(\cdot)P(\cdot)\nabla u^{H}(x)\Vert_{L^\infty(Y)}
\end{eqnarray}
and we conclude that
\begin{eqnarray}
\lim_{k \to \infty }\|\chi^i(n_k x)P(n_{k}x) \nabla u^{H}(x) \|_{L^{\infty} (\Omega ')}\leq\Vert{\mathcal M}^i(\nabla u^H)\Vert_{L^\infty(\Omega')}.
\label{cuniformconvgiubddd}
\end{eqnarray}
The lower bound
\begin{eqnarray}
\Vert{\mathcal M}^i(\nabla u^H)\Vert_{L^\infty(\Omega')}\leq\lim_{k\to \infty}\|\chi^i(n_k x)\nabla u_{n_{k}}\|_{L^{\infty} (\Omega '))}
\label{cuniformconvgilbddd}
\end{eqnarray}
follows from a direct application of Corollary \ref{inftyonesidedinequality} proved the next section and we conclude that
\begin{eqnarray}
\lim_{k\to \infty}\|\chi^i(n_k x)\nabla u_{n_{k}}\|_{L^{\infty} (\Omega '))} =\Vert{\mathcal M}^i(\nabla u^H)\Vert_{L^\infty(\Omega')}.
\label{cuniformconvgie}
\end{eqnarray}
The theorem follows on noting that identical arguments can be applied to every subsequence of $\{\chi^i(nx)\nabla u_n\}_{n=1}^\infty$ to conclude the existence of a further subsequence with limit given by \eqref{cuniformconvgie}.

\end{proof}
\setcounter{theorem}{0}
\setcounter{definition}{0}
\setcounter{lemma}{0}
\setcounter{conjecture}{0}
\setcounter{corollary}{0}
\setcounter{equation}{0}
\section{Continuously graded microstructures}
In this section we consider a class of coefficient matrices associated with {\em continuously graded} composites made from  N distinct materials.
In order to express the continuous gradation of the microstructure we introduce the characteristic functions $\chi^{i}(x, y)$, $i = 1,\dots, N$ belonging to  $L^1(\Omega \times Y)$ such that for each $x$ the function $\chi^{i}(x, \cdot)$ is periodic and represents the characteristic function of the $i^{th} $ material inside the unit period cell $Y$.  The characteristic functions are taken to be continuous in the $x$ variable according to the following continuity condition given by
\begin{equation}\label{continuity}
\lim_{h \to 0}\int_Y |\chi^i(x+h,y)-\chi(x,y)|\,dy=0.
\end{equation} 
The coefficient associated with each material is denoted by $A_i$ and is a constant symmetric matrix satisfying the ellipticity condition
\[
\lambda \leq A_{i}\leq\Lambda
\]
for fixed positive numbers $\lambda < \Lambda$.
We define the coefficient matrix 
\[
A(x,y) = \sum_{i=1}^{N}A_{i}\chi^{i}(x,y).
\]
This type of coefficient matrix appears in prototypical problems where one seeks to design structural components made from  functionally graded materials \cite{markworth} and \cite{ootao}.  Here the configuration of the N materials is locally periodic but changes across the domain $\Omega$. The composite is constructed by dividing the domain $\Omega$  into subdomains $\Omega_{k,l}$, $l=1,\ldots,M_k$  of diameter less than or equal to $1/k$, $k=1,2,\ldots$ and $\Omega=\cup_{l=1}^{M_k}\Omega_{k,l}$. Each subdomain contains a periodic configuration of $N$ materials. 
The following lemma allows us to approximate the ideal continuously graded material by a piecewise periodic functionally graded material that can be manufactured.
\begin{lemma}
For a given subdivision $\Omega_{k,1},\ldots\Omega_{k,M_k}$ of diameter less than $1/k$ and any $i = 1,\cdots, N$,  there exists a sequence $\{\chi^i_k(x,y)\}_{k=1}^\infty$ of approximations to $\chi^i(x,y)$ given by
\begin{equation}
\chi^i_k(x,y)=\sum_{l}\chi_{\Omega_{k,l}}(x)\chi^i_{k,l}(y)
\label{graded}
\end{equation}
with the property that 
\begin{equation}
\lim_{k\to \infty}\int_{\Omega\times Y}|\chi^i_k(x,y)-\chi^i(x,y)|dydx=0. 
\label{approx}
\end{equation} 
In \eqref{graded}, $\chi_{\Omega_{k,l}}(x)$ denotes the characteristic function of $\Omega_{k,l}$ and $\chi^i_{k,l}(y) = \chi^{i}(x_{k,l},y)$ is the characteristic function associated with the configuration of the $i^{th}$ phase inside the subdomain $\Omega_{k,l}$ at a sample point $x_{k,l} \in \Omega_{k,l}$.
\label{cgraded}
\end{lemma}
\begin{proof}
The definition of the approximating function is given in \eqref{graded}. We verify that \eqref{approx} is satisfied. For each $x\in \Omega,$  define the sequence of functions 
\[
\Gamma^{i}_{k}(x)  = \int_{Y}|\chi^i_k(x,y)-\chi^i(x,y)|dy. 
\]
Then $\Gamma^{i}_{k}(x)\to 0$ for all $x\in \Omega.$  Indeed, for a fixed $x\in \Omega, $ there exists a sequence of subdomains $x\in \Omega_{k, l_{k}}$ and points $x_{k, l_{k}}\in \Omega_{k, l_{k}}$ such that by definition,
\[
\Gamma_{k}^{i} (x) = \int_{Y}|\chi^{i}(x_{k, l_{k}}, y) - \chi^{i}(x, y)|dy.
\]
It is evident that  $|x_{k, l_{k}} - x| < 1/k$ since $x_{k, l_{k}}$ and $x$  both belong to $\Omega_{k, l_{k}}.$ Applying the continuity condition \eqref{continuity}, we see that $\Gamma_{k}^{i} (x) \to 0$ as $k\to \infty$ and \eqref{approx} follows from the Lebesgue dominated convergence theorem. 
\end{proof}

Let us define the coefficient matrix of the functionally graded material. Divide the domain $\Omega$  into subdomains $\Omega_{k,l}$, $l=1,\ldots,M_k$  of diameter less than or equal to $1/k$, $k=1,2,\ldots$ and $\Omega=\cup_{l=1}^{M_k}\Omega_{k,l}$. Each subdomain contains a periodic configuration of $N$ materials with period $1/n$ such that  $1/k>1/n$. The configuration of the $i^{th}$ phase inside a {\em functionally graded} composite is described by $\chi^i_k(x,nx)$, where $\chi_{k}^{i}(x,y)$ is given by \ref{graded}. The corresponding coefficient matrix is denoted by $A^k(x,nx)$ and is written as
\begin{equation}
A^k(x, nx) = \sum_{i}^{M_k}\chi^i_{k}(x, nx)A_{i}.
\label{coeffk}
\end{equation}
As seen from the proof of the lemma the continuity condition \eqref{continuity} insures that near by subdomains $\Omega_{k,l}$ and $\Omega_{k,l'}$ have configurations that are nearly the same
when $1/k$ is sufficiently small.
The fine-scale limit of such composites is obtained by considering a family of partitions indexed by $j=1,2, \dots, $ with subdomains $\Omega_{l}^{k_{j}}$ of diameter less that or equal to $1/k_{j}$. The scale of the microstructure is given by $1/n_{j}$. Both $1/k_{j}$ and $ 1/n_{j} $ approach zero as $j$ goes to infinity and  we require that $\lim_{j \to \infty}\frac{1/n_{j}}{1/k_{j}}=0$. For future reference the associated indicator functions and coefficients are written
\begin{equation}
  \chi^i_{k_j}(x, n_j x)~~\text{and} ~~A^{k_j}(x,n_j x).
\label{charcoeff}
\end{equation}
Let
\begin{equation}\label{effective}
A^{H}(x) = \int_{Y}A(x, y)P(x, y)dy
\end{equation}
where the matrix $P(x,y)$ is defined by
\begin{equation}
P(x, y)_{i, j} = \frac{\partial w^{j}}{\partial y_{i}} + \delta_{i j},
\label{pcontinuous}
\end{equation}
and $w^{i}(x, \cdot)$ is a $Y$ periodic function that solves the PDE
\begin{equation}\label{cellgraded}
\text{div}_{y}(A(x, y)(\nabla_{y}w^{i}(x, y) + e^{i}) )= 0,
\end{equation}
where  $\{e^{i}\}$, $i=1,\ldots$ is an orthonormal basis for $\mathbb{R}^{d}$.

The Sobolev space of  square integrable functions with square integrable derivatives periodic on $Y$ is denoted by
 $H_{\rm{per}}^{1}(Y)$. The functions $w^i(x,y)$ belong to $C(\Omega;H^{1}_{\rm{per}}(Y))$ this follows from \eqref{continuity} and is proved in  the Appendix.

We present the homogenization theorem for the sequences $A^{k_j}(x,n_j x)$ proved in \cite{liproy}.

\begin{lemma}\label{G-convergence}
One can construct sequences  $\{\chi^i_{k_j}(x,xn_j)\}_{j=1}^\infty$ for which the coefficient matrices $\{ A^{k_{j}}(x, n_{j}x)\}_{j=1}^{\infty}$  $G-$ converge to the effective  tensor $A^{H}(x)$ defined by (\ref{effective}).
\end{lemma}

Let  $f\in H^{-1}_{0}(\Omega)$ be given. Then by Lemma \ref{G-convergence}
the sequence of solutions$\{u_{j}\}$ of the equation:
\[
- \text{div}[A^{k_{j}}(x, n_{j}x)\nabla u_{j} (x)] = f, \quad u_{j} \in H^{1}_{0}(\Omega)
\]
converge  to $u^{H}$ weakly in $H^{1}_{0}$, where $u^{H}$ solves the the equation
\[
-\text{div} [A^{E}(x))\nabla u^{H} (x)] = f, \quad u^{H} \in H^{1}_{0}(\Omega).
\]
We now have the following result.
\begin{theorem}\label{representationThm} Let $V \subset L^{1}(Y)\cap L^{\infty}(Y)$  be a countable  dense subset of $L^{1}(Y)$. Assume that all elements of $V$ are periodically extended to $\mathbb{R}^{d}$.
Suppose that  $\phi(x)\in C(\overline{\Omega})$, $\eta(x) \in V$
 and $u_{j}$, $P$ and $u^{H}$ be given as above. Then
\[
\lim_{j\to \infty}\int_{\Omega}\phi(x)\eta( n_{j}x)\chi^{i}_{k_j}(x, n_{j}x)|\nabla u_{j}|^{2}dx = \int_{\Omega}\int_{Y}\phi(x)\eta(y)\chi^{i}(x,y)|P(x, y)\nabla u^{H}(x)|^{2}dydx.
\]
\end{theorem}
By taking the modulation functions for continuously graded composites (see, \cite{liproy}) to be 
\begin{eqnarray}
{\mathcal M}^i(\nabla u^H)(x)=\|\chi^{i}(x, \cdot)P(x,\cdot)\nabla u^{H}(x) \|_{L^\infty(Y)}.
\label{figraded}
\end{eqnarray}
we obtain the following corollary.
\begin{corollary}\label{inftyonesidedinequality}
\begin{eqnarray}
\|{\mathcal M}^i(\nabla u^H)\|_{L^{\infty}(\Omega)}\leq \limsup_{j\to \infty}\|\chi^{i}_{k_j}(x, n_j x)\nabla u_{j}\|_{L^{\infty}(\Omega)}.\label{lowergraded}
\end{eqnarray}
\end{corollary}
Under an additional asymptotic condition on the distribution functions for the sequence $\{\nabla u_j\}_{j=1}^\infty$, equality can be achieved in the above corollary. Indeed, define
\begin{eqnarray}
S_{t, i}^{j}=\{ x: \chi^{i}_{k_j}(x, n_j x)|\nabla u_{j}(x)|^{2} > t\}, \quad \chi_{t, i}^{j}(x) := \chi_{S_{t, i}^{j}}
\label{setsindicator}
\end{eqnarray}
and the distribution functions are given by
\begin{eqnarray}
|S_{t,i}^j|=\int_\Omega\chi_{t, i}^{j}\,dx.
\label{distribution}
\end{eqnarray}
Passing to a subsequence, there exists density functions $\theta_{t, i}$ such that
\[
\chi_{t, i}^{j} (x) \stackrel{*}{\rightharpoonup} \theta_{t, i}(x) \quad \text{$L^{\infty} $weak *}
\]
and for any open subset $S\subset \Omega$
\begin{eqnarray}
\lim_{j\to\infty}|S_{t,i}^j\cap S| = \int_S\theta_{t,i}\,dx.
\end{eqnarray}
We present a sufficient condition on the distribution functions $|S_{t,i}^j|$ associated with $\{\nabla u_j\}_{j=1}^\infty$ for which equality
holds in \eqref{lowergraded}.
\begin{corollary}\label{inftyEquality}
Suppose that $l = \limsup_{j\to \infty}\|\chi^{i}_{k_j}(x, n_j x)|\nabla u_{j}|^{2}\|_{L^{\infty}(\Omega)}< \infty$ and for each $\delta>0$ there exists a positive number $\beta_\delta>0$ for which
\begin{eqnarray}
|\{x\in\Omega:\theta_{l-\delta, i} > 0\}|>\beta_\delta.
\label{suffconddd}
\end{eqnarray}
Then
\[
\limsup_{j\to \infty}\|\chi^{i}_{k_j}(x, n_{j}x)\nabla u_{j}\|_{L^{\infty}(\Omega)} = \| {\mathcal M}^i(\nabla u^H)\|_{L^{\infty}(\Omega)}
\]
\end{corollary}
\begin{proof}(Proof of Corollary \ref{inftyEquality})
The homogenization constraint \cite{liproy} states that
for almost every $x\in \Omega$
\[
\theta_{t, i}(x)({\mathcal M}^{i}(\nabla u^H(x)) - t)\geq 0\quad i=1, \dots N.
\]
It follows that  on the set where $\theta_{t, i} >0$, we have ${\mathcal M}^{i}(\nabla u^H(x)) \geq t$.
Let \[l_{j} = \|\chi^{i}_{k_j}(x, n_{j}x)|\nabla u^{j}|^{2}\|_{L^{\infty}(\Omega)}.\]
For a subsequence $l_{j} \to l$. Then given $\delta > 0$, there exists a natural number $J$ such that
\[
l-\delta/2< l_{j} = \|\chi^{i}_{k_j}(x, n_{j}x)|\nabla u^{j}|^{2}\|_{L^{\infty}(\Omega)} < l+\delta/2 \quad \forall j\geq J.
\]
The measure of the set $S_{l_{j}-\delta/2, i}^{j}$ is positive. Moreover $S_{l_{j}-\delta/2, i}^{j} \subset S_{l-\delta, i}^{j} $ and
\[
\chi_{l-\delta, i}^{j}\stackrel{*}{\rightharpoonup}\theta_{l-\delta, i}(x)\quad \text{$L^{\infty}$ weak *~ as $j\to \infty$ }
\]
 From hypothesis the set where $\theta_{l-\delta, i} > 0$ is a set of positive measure for all $\delta > 0$. Therefore,
\[
{\mathcal M}^{i}(\nabla u^H(x))\geq l-\delta,
\]
on a set of positive measure that is
$
\|{\mathcal M}^{i}(\nabla u^H(x))\|_{\infty}\geq l-\delta.
$
The corollary is proved since $\delta >0$ is arbitrary.
\end{proof}

\begin{proof}(Proof of corollary \ref{inftyonesidedinequality})
 From Theorem \ref{representationThm} it follows that for any $\phi\in C(\overline{\Omega})$ and $\eta \in V$ is $Y-$ periodic,
\[
\begin{split}
\int_{\Omega}\int_{Y}\chi^{i}(x, y)&\phi(x)\eta( y)|P(x,y)\nabla u^{H}|^{2}dydx \\&\leq \lim_{j\to \infty}\int_{\Omega}|\phi(x)\eta( n_{j}x)|dx \limsup_{j\to \infty}\|\chi^{i}_{k_j}(x, n_{j}x)|\nabla u_{j}|^{2}\|_{L^{\infty}(\Omega)}
\end{split}
\]
By the Riemann-Lebesgue lemma,
\[
\lim_{j\to \infty}\int_{\Omega}|\phi(x)\eta( n_{j}x)|dx = \int_{\Omega}\phi(x)dx\int_{Y}|\eta( y)|dy.
\]
Dividing both sides  by the $L^{1}$-norm of $\phi$, we obtain that for every $x\in \Omega \setminus Z$, where $Z$ is a set of measure zero,
\[
\int_{Y}\chi^{i}(x, y)\eta( y)|P(x,y)\nabla u^{H}|^{2}dy \leq \int_{Y}|\eta( y)|dy\limsup_{j\to \infty}\|\chi^{i}_{k_j}(x, n_{j}x)|\nabla u_{j}|^{2}\|_{L^{\infty}(\Omega)}
\]
The set $Z$  depends on the choice of $\eta$. But since $V$ is countable, the union of the sets  $Z$ corresponding to elements of $V$ will be of measure zero and the above inequality is true for any $\eta \in V$ and for every $x$ outside this union. Now divide the last inequality by the $L^{1}$ norm of $\eta$ in $Y$. Taking the sup over $V$ and noting that $V$ is dense in $L^{1}(Y)$, proves the corollary.
\end{proof}
\begin{proof} (Proof of Theorem \ref{representationThm})
For $\beta > 0$ , define
\begin{eqnarray}
A_{1}(x, y) = A(x, y) + \beta\chi^{i}(x, y)\phi(x)\eta(y)I.
\label{coeffpurt}
\end{eqnarray}
Now let $v_{j}$ solve
\[
-\text{div} [A^{k_{j}}_{1}(x, n_{j}x)\nabla v_{j}]=f, \quad v_{j} \in H_{0}^{1}(\Omega)
\]
Then for any $\varphi\in H_{0}^{1}(\Omega)$, we have
\begin{equation}\label{integralform}
\begin{split}
\int_{\Omega} (A_{1}^{k_{j}}(x, n_{j}x)\nabla v_{j}, \nabla \varphi)dx &= \int_{\Omega} f\varphi dx \quad \text{and }\\
\int_{\Omega} (A^{k_j}(x, n_{j} x)\nabla u_{j}, \nabla \varphi)dx &= \int_{\Omega} f\varphi dx
\end{split}
\end{equation}
Let $\delta u_{j} = v_{j} - u_{j}$. Then subtracting the second equation above from the first, we obtain
\[
\int_{\Omega} (A_{1}^{k_{j}}(x, n_{j}x)\nabla\delta u_{j}, \nabla \varphi)dx + \int_{\Omega}((A_{1}^{k_{j}}( x, n_{j}x) - A^{k_{j}}(x, n_{j}x))\nabla u_{j}, \nabla \varphi )dx = 0,
\]
for all $\varphi\in H_{0}^{1}(\Omega)$. Simplifying the above equation we get
\[
\int_{\Omega} (A_{1}^{j}(x, n_{j}x)\nabla\delta u_{j}, \nabla \varphi)dx + \beta\int_{\Omega}\chi^{i}_{k_j}(x, n_{j}x)\phi(x)\eta( n_{j}x)(\nabla u_{j}, \nabla \varphi )dx = 0,
\]
Plug in $\varphi = u_{j}$ in the above equation to get,
\begin{equation}\label{integralform2}
\int_{\Omega} (A_{1}^{k_{j}}(x, n_{j}x)\nabla\delta u_{j}, \nabla u_{j})dx +\beta \int_{\Omega}\chi^{i}_{k_{j}}(x, n_{j}x)\phi(x)\eta(n_{j}x)|\nabla u_{j}|^{2}dx = 0,
\end{equation}
Also plugging in $\varphi  = \delta u_{j}$ in (\ref{integralform}) yields
\begin{equation}\label{integralform3}
\int_{\Omega} (A^{k_j}(x, n_{j}x)\nabla u_{j}, \nabla \delta u_{j})dx = \int_{\Omega} f\delta u_{j} dx.
\end{equation}
Subtracting (\ref{integralform3}) from (\ref{integralform2}) and noting that the coefficient matrices are symmetric we get
\[
\beta\int_{\Omega}\chi^{i}_{k_{j}}(x, n_{j}x)\phi(x)\eta(n_{j}x)|\nabla u_{j}|^{2}dx + T^{j} = -\int_{\Omega} f\delta u_{j} dx
\]
where
\[
T^{j} 
= \beta \int_{\Omega}\chi_{i}^{k_j}(x, n_{j}x)\phi(x)\eta(n_{j}x)(\nabla u_{j}, \nabla \delta u_{j})dx.
\]
Let us estimate $T^{j}.$ To begin with, observe that
\[
\int_{\Omega} (A_{1}^{k_j}(x, n_{j}x)\nabla\delta u_{j}, \nabla \delta u_{j})dx + \beta\int_{\Omega}\chi^{i}_{k_j}(x, n_{j}x)\phi(x)\eta(n_{j}x)(\nabla u_{j}, \nabla \delta u_{j} )dx = 0,
\]
Them from ellipticity, we get
\[
\begin{split}
\alpha \int_{\Omega}|\nabla \delta u_{j}|^{2}dx&\leq \int_{\Omega} (A_{1}^{k_j}(x, n_{j}x)\nabla\delta u_{j}, \nabla \delta u_{j})dx\\
&\leq \beta \int_{\Omega}\chi^{i}_{k_j}(x, n_{j}x)|\phi(x)\eta(n_{j}x)||\nabla u_{j}|| \nabla \delta u_{j} |dx\\
&\leq C\beta \|\nabla \delta u^{j}\|_{L^2}\|\nabla u^{j}\|_{L^2}.
\end{split}
\]
That is
\[
\|\nabla \delta u_{j}\|_{L^2} \leq C\beta,
\]
since the sequence $\nabla u^{j}$ is bounded in $L^2$. From this and the definition on $T^{j}$ we obtain
\[
|T^{j}| \leq C\beta^{2}
\]
From Lemma \ref{G-convergence} we know that $u_{j} \rightharpoonup u^{H}$, and $v_{j} \rightharpoonup v^{H}$, where $u^{H}$ and $v^{H}$ satisfy the following equations, respectively:  For any  $\varphi\in H_{0}^{1}(\Omega)$
\begin{equation}\label{homoguv}
\begin{split}
 \int_{\Omega} A_{1}^{H}(x)\nabla v^{H}, \nabla \varphi)dx &= \int_{\Omega}f\varphi dx\\
\int_{\Omega} A^{H}(x)\nabla u^{H}, \nabla \varphi)dx &= \int_{\Omega}f\varphi dx
\end{split}
\end{equation}
where $A^{H} (x)$, is the effective matrix given by \eqref{effective} and
\begin{eqnarray}
A_1^{E}(x) = \int_{Y}A_1(x, y)P_1(x, y)dy
\label{effpert}
\end{eqnarray}
where the matrix $P_1(x,y)$ is defined by
\begin{eqnarray}
P_1(x, y)_{i, j} = \frac{\partial w_1^{j}}{\partial y_{i}} + \delta_{i j},
\end{eqnarray}
and $w_1^{i}(x, \cdot)$ is a $Y$ periodic function that solves the PDE
\begin{eqnarray}
\text{div}_{y}(A_1(x, y)(\nabla_{y}w_1^{i}(x, y) + e^{i}) )= 0,\label{cellgradedpert}
\end{eqnarray}
where  $\{e^{i}\}$, $i=1,\ldots$ is an orthonormal basis for $\mathbb{R}^{d}$.

Writing $\delta u^H=v^H-u^H$  and letting $j\to \infty$, we obtain
\begin{equation}\label{firstway}
\begin{split}
\beta\lim_{j\to\infty}\int_{\Omega}\chi^{i}_{k_j}(x, n_{j}x)\phi(x)\eta(n_{j}x)|\nabla u_{j}|^{2}dx + \lim_{j\to \infty}T^{j} &= -\lim_{j\to \infty}\int_{\Omega} f\delta u_{j} dx\\
&=-\int_{\Omega} f\delta u^H dx.
\end{split}
\end{equation}
One easily verifies that the variational formulations \eqref{homoguv} can be written in terms of the two scale variational principles \cite{Nguetseng}, \cite{Allaire}
given by
\begin{equation}\label{twohomoguv}
\begin{split}
 \int_{\Omega}\int_Y A_{1}(x,y)(\nabla v^{H}(x)+\nabla_y v_1(x,y)), \nabla \varphi(x)+\nabla_y\varphi_1(x,y)dydx &= \int_{\Omega}f\varphi dx\\
\int_{\Omega}\int_Y A(x,y)(\nabla u^{H}(x)+\nabla_y u_1(x,y)), \nabla \varphi(x)+\nabla_y \varphi_1(x,y)dydx &= \int_{\Omega}f\varphi dx,
\end{split}
\end{equation}
where the solutions $(u^H,u_1)$, $(v^H,v_1)$, and trial fields $(\varphi,\varphi_1)$ belong to the space $H_0^{1}(\Omega)\times L^2[\Omega;W_{\rm{per}}^{1,2}(Y)]$.
On writing $\delta u_1=v_1-u_1$, $\delta u^H=v^H-u^H$, $A_1(x,y)=A(x,y)+\beta\chi^i(x,y)\phi(x)\eta(y)I$,  substitution into
the first equation in \eqref{twohomoguv} and applying the second equation in \eqref{twohomoguv} gives
\begin{equation}\label{ident}
\begin{split}
 &\int_{\Omega}\int_Y A_{1}(x,y)(\nabla \delta u^{H}(x)+\nabla_y \delta u_1(x,y)), \nabla \varphi(x)+\nabla_y\varphi_1(x,y)dydx \\
&+\beta\int_{\Omega}\int_Y (\chi^i(x,y)\phi(x)\eta(y)(\nabla u^{H}(x)+\nabla_y u_1(x,y)), \nabla \varphi(x)+\nabla_y \varphi_1(x,y)dydx \\
&=0.
\end{split}
\end{equation}
Next we substitute $(\varphi,\varphi_1)=(\delta u^H,\delta u_1)$ into the second equation of \eqref{twohomoguv} to obtain the identity
\begin{equation}\label{secondidentity}
\begin{split}
&\int_{\Omega}\int_Y A(x,y)(\nabla u^{H}(x)+\nabla_y u_1(x,y)), \nabla \delta u^H(x)+\nabla_y \delta u_1(x,y)dydx \\
&= \int_{\Omega}f(x)\delta u^H(x) dx.
\end{split}
\end{equation}
On choosing $(\varphi,\varphi_1)=(u^H,u_1)$ in \eqref{ident} and applying \eqref{secondidentity} we obtain
\begin{equation}\label{Finaltident}
\begin{split}
&T+\beta\int_{\Omega}\int_Y(\chi^i(x,y)\phi(x)\eta(y)(\nabla u^{H}(x)+\nabla_y u_1(x,y)), \nabla \varphi(x)+\nabla_y \varphi_1(x,y))dydx\\
&=-\int_\Omega f\delta u^H dx,
\end{split}
\end{equation}
where
\begin{equation}\label{Tident}
T=\beta\int_{\Omega}\int_Y(\chi^i(x,y)\phi(x)\eta(y)(\nabla u^{H}(x)+\nabla_y u_1(x,y)), \nabla \delta u^H(x)+\nabla_y \delta u_1(x,y))dydx.
\end{equation}
Next we set $(\varphi,\varphi_1)=(\delta u^H,\delta u_1)$ in \eqref{ident} and applying ellipticity delivers the estimate
\begin{equation}
\Vert\nabla \delta u^H+\nabla_y\delta u_1\Vert_{L^2(\Omega\times Y)}\leq C\beta,
\label{deltaestimate}
\end{equation}
and we find that
\begin{equation}
\label{Tbound}
|T|\leq C\beta^2.
\end{equation}
Since \eqref{firstway}  and \eqref{Finaltident} have the same right hand sides we equate them and the theorem follows on identifying like powers of $\beta$.
\end{proof}
We conclude this section noting that lower bounds similar to those given here can be obtained in the context of two-scale convergent coefficient matrices  \cite{lipsima}.

\setcounter{theorem}{0}
\setcounter{definition}{0}
\setcounter{lemma}{0}
\setcounter{conjecture}{0}
\setcounter{corollary}{0}
\setcounter{equation}{0}
\section{Local representation formulas and gradient constrained design for graded materials}

In view of  applications it is important to identify
graded material properties that deliver a desired level of
structural performance while at the same time provide a hedge against failure initiation \cite{gosschristensen}. In many applications there is a separation
of scales and the material configurations forming up the microstructure
exist on length scales significantly smaller than the characteristic
length scale of the loading. Under this hypothesis the structural properties are
modeled using effective thermophysical
properties that depend upon features of the underlying
micro-geometry, see \cite{fuji}, \cite{markworth}. In this context overall structural performance measured by resonance frequency and  structural stiffness are recovered from the solutions of homogenized equations given in terms of the effective coefficients (G-limits). In order to go further and design against failure initiation we record the effects of $L^\infty$ constraints on the local gradient field inside functionally graded materials. For this we use the local representation formulas given by modulation functions \eqref{figraded}.

The  multi-scale formulation of the graded material design problem has three features \cite{Liptstueb3}, \cite{LiptstuebAIAA}:
\begin{enumerate}

\item It admits a convenient local parametrization of microstructural information expressed in terms of a homogenized coefficient matrix
\eqref{effective} and  local representation formulas given by the modulation functions \eqref{figraded}.

\item  Is well posed, i.e., an optimal design exists.

\item The optimal design is used to identify an explicit  ``functionally graded microstructure'' that delivers
an acceptable level of structural performance while controlling the local gradient field over a predetermined part of the structural domain.

\end{enumerate}

\noindent In what follows we  describe the multiscale material design problem and focus the discussion on the control of the $L^\infty$ norm of the local
gradient field.

The admissible set of continuously graded locally periodic
microstructures is specified by a vector
$\underline{\beta}=(\beta_1,\ldots,\beta_n)$ of local geometric
parameters. For example one may consider a periodic array of spheroids described
by the orientation of their principle axis and aspect ratio. The
periodic microstructure is specified in a unit period cell $Y$
centered at the origin. Points in the cell are denoted by $y$.
The characteristic function of the $i^{th}$ phase in the unit cell
is denoted by $\chi^i(\underline{\beta},y)$, $i=1,\ldots,N$.  The vector $\underline{\beta}$
for the graded microgeometry can change across the
structural domain $\Omega$ and we write $\underline{\beta}=\underline{\beta}(x)$ for $x\in\Omega$. The $x$ dependence of $\underline{\beta}$  corresponds to the
gradation of material properties through a gradation in
microstructure. The design vector $\underline{\beta}(x)$ is a uniformly H\"older continuous function of $x$ in the closure of $\Omega$.
We write
\begin{eqnarray}
\chi^i(x,y)=\chi^i(\underline{\beta}(x),y)
\label{identgrade}
\end{eqnarray}
and since $\underline{\beta}(x)$ is continuous one sees that $\chi^i(x,y)$ is continuous in the sense of \eqref{continuity}.

The multi-scale design problem is formulated as follows:
The admissible set $Ad$ of design vectors $\underline{\beta}(x)$ is the set of uniformly H\"older continuous functions satisfying the two conditions:

\begin{itemize}
\item
There is a fixed positive constant $C$
such that:
\begin{eqnarray}
\sup_{x,x'\in\overline{\Omega}}\frac{|\underline{\beta}(x)-\underline{\beta}(x')|}{|x-x'|}<C.
\label{holder}
\end{eqnarray}
\item The design vector $\underline{\beta}(x)$ takes values inside the closed bounded set given by the constraints
\begin{eqnarray}
\underline{b}_i\leq\beta_i(x)\leq\overline{b}_i,\hbox{ $i=1,\ldots,n$.}
\label{box}
\end{eqnarray}
\end{itemize}

The local volume fraction of the $i^{th}$ phase in the composite
is given by $\theta_i(x)=\int_Y\chi^i(x,y)\,dy$.
A resource constraint is placed on the amount of each phase
appearing the design. It is given by
\begin{equation}
\int_\Omega\,\theta_i(x)\,d x\leq \gamma_i,\,\,i=1,\ldots,N.
\label{constr2}
\end{equation}
The vector of constraints $(\gamma_1,\ldots,\gamma_N)$ is denoted
by $\underline{\gamma}$. The set of controls $\underline{\beta}(x)\in Ad$ that
satisfy the resource constraints
(\ref{constr2}) is denoted by ${\mathcal A}d_{\underline{\gamma}}$.

As an example we assume homogeneous Dirichlet conditions on the boundary of the design domain $\Omega$.
For a given right hand side $f\in H^{-1}(\Omega)$
the overall structural performance of the graded composite is modeled using the solution $u^H$ of
the homogenized equilibrium equation given by the
$H_0^1(\Omega)$ solution of
\begin{equation}
-{\rm div\,}\left(A^H(x)\nabla u^H\right)=f.
\label{equlibelastth}
\end{equation}
Here $A^H$ is given by \eqref{effective} with $\chi_i(x,y)$ given by \eqref{identgrade}.

In this example the overall work done against the load  is used as the performance measure of the
graded material structure. This functional depends nonlinearly on the design $\underline{\beta}$ through the solution $u^H$ and is given by
\begin{eqnarray}
W(\underline{\beta})=\int_\Omega\,f u^H\,dx,
\label{objective}
\end{eqnarray}

We pick an open subset $S\subset\Omega$ of interest and the gradient constraint for the multi-scale problem is written in terms of
the modulation function. We set
\begin{eqnarray}
C_i(\underline{\beta})=\Vert {\mathcal M}^i(\nabla u^H)\Vert_{L^\infty(S)}, \hbox{ for $i=1,\ldots,N$}
\label{gradconstmulti}
\end{eqnarray}
and the multi-scale optimal design problem is given by
\begin{equation}
P=\inf_{\underline{\beta}\in {\mathcal A}d_{\underline{\gamma}}} \{W(\underline{\beta}):\, C_i(\underline{\beta})\leq M,\,i=1,\ldots,N\}.
\label{basicopth}
\end{equation}
When the constraint $M$ is chosen such that there exists a control $\underline{\beta}\in{\mathcal A}d_{\underline{\gamma}}$ for which $C_i(\underline{\beta})\leq M$ then an optimal design $\underline{\beta}^*$ exists for the design problem (\ref{basicopth}),
this is established in \cite{lipnato}, \cite{Liptstueb1}.

The optimal design $\underline{\beta}^*$ specifies characteristic functions $\chi^{i^*}(x,y)=\chi^i(\underline{\beta}^*(x),y)$
from which we recover continuously graded microgeometries   $\chi^{i^*}_{k_j}\left(x,n_j x\right)$ and  coefficient matrices\\
$A^{*,k_j}(x,n_j x)$ of the form \eqref{charcoeff}. The coefficients $A^{*,k_j}\left(x,n_j x\right)$ G-converge to the effective coefficient $A^{H^*}$ associated with the optimal design $\underline{\beta}^*$, see Lemma \ref{G-convergence}. Here the effective coefficient is
given by \eqref{effective} with $\chi^i(x,y)=\chi^{i^*}(x,y)$. For each $j=1,\ldots$ the $H_0^1(\Omega)$ solution $u_j$ of the equilibrium problem inside the graded composite satisfies
\begin{eqnarray}
-{\rm div}\left(A^{*,k_j}(x,n_j x)\nabla u_j\right)=f
\label{eqstar}
\end{eqnarray}
and the work done against the load is given by $W(u_j)=\int_{\Omega} f u_j dx$.
This functional is continuous with respect to G-convergence
hence $\lim_{j \to \infty} W(u_j)=W(\underline{\beta}*)$.

We now  apply Theorem \ref{upperboundd1} to discover that for any open set $S\subset\Omega$ with closure contained inside $\Omega$
 there exists  a decreasing sequence of sets $E_{k_j}$ for which
$|E_{k_j}|\searrow 0$ and
\begin{eqnarray}
\limsup_{j\rightarrow\infty}\Vert\chi^{i^*}_{k_j}(x,n_j x)\nabla u_j(x)\Vert_{L^\infty(S\setminus E_{k_j})}\leq M,\hbox{ $i=1,2,\ldots,M$.}\label{ubound1applied}
\end{eqnarray}
Therefore we can choose a graded material design specified by $\chi^{i^*}_{k_j}(x, n_j x)$ with overall structural properties $W(u_j)$ close to
the optimal one  $W(\beta^*)$ and with
$$\Vert\chi^{i^*}_{k_j}(x,n_j x)\nabla u_j(x)\Vert_{L^\infty(S\setminus E_{k_j})}\leq M$$
outside controllably small sets $E_{k_j}$.
This is the essence of the design scheme for continuously graded composite structures developed in \cite{Liptstueb1}, \cite{LiptstuebAIAA}.

We conclude this section with a conjecture.
Numerical simulations \cite{LiptstuebAIAA} show that when the microstructure corresponds to smooth inclusions embedded inside a matrix, such as shafts reinforced with long prismatic fibers with circular cross section, then the design method implies full control of the local gradient over the set $S$ i.e.,
\begin{eqnarray}
\Vert\chi^{i^*}_{k_j}(x,n_j x)\nabla u_j(x)\Vert_{L^\infty(S)}\leq M,\hbox{ $i=1,2,\ldots,M$.}\label{ubound1appliedexample}
\end{eqnarray}
With this in mind and in view of Theorem \ref{periodicidentity} we are motivated to propose the following conjecture.

\begin{conjecture}
For continuously graded composites containing inclusions with $C^{1,\alpha}$ boundaries for which  $A^{k_j}(x,n_j)$ G-converges to $A^H(x)$  then
\begin{eqnarray}
\limsup_{j\to \infty}\|\chi^{i}_{k_j}(x, n_{j}x)\nabla u_{j}\|_{L^{\infty}(S)} = \Vert {\mathcal M}^i(\nabla u^H)\Vert_{L^\infty(S)}.
\label{eq2conj}
\end{eqnarray}
\end{conjecture}
\setcounter{equation}{0}
\section*{Appendix }
Here we will show that the solutions $w^{i}$ of the cell problem \eqref{cellgraded} satisfying $\int_{Y} w^{i}(x, y)dy  = 0$ are in $C(\Omega, H^{1}_{per}(Y) )$ under the continuity assumption \eqref{continuity}. 
To that end, it suffices to show that as $h\to 0$
\[
\|\nabla_{y}w^{i}(x+h, \cdot) - \nabla _{y}w^{i}(x, \cdot)\|_{L^{2}(Y)}\to 0.
\]
Since $w^{i}(x+h, y)$ solves equation  \eqref{cellgraded} $A(x,y)$ replaced by $A(x+h, y)$, we have that
\[
\text{div}~(A(x+h,y)(\nabla_{y}w^{i}(x+h, y) + e^{i}))= \text{div}~(A(x,y)(\nabla_{y}w^{i}(x, y) + e^{i}))= 0.
\]
Rewriting the above equation we obtain
\[
\text{div}~[(A(x+h,y)-A(x,y))](\nabla_{y}w^{i}(x+h, y) + e^{i}) = \text{div}~(A(x,y)(\nabla_{y}w^{i}(x, y) -\nabla_{y}w^{i}(x+h, y)).
\]
Define the difference mapping $\delta_{h}(F) = F(x+h, y)-F(x,y)$.
Then for any $\psi\in H^{1}_{per}(Y)$, we have
\begin{equation}\label{continuityDiff}
-\int_{Y}(A(x,y)\nabla_{y}\delta_{h}(w^{i})(x, y), \nabla \psi)dx = \int_{Y}(\delta_{h}(A)(x,y)(\nabla_{y}w^{i}(x+h, y) + e^{i}), \nabla \psi).
\end{equation}
Then plugging $\psi(x,y) = \delta_{h}(w^{i})(x, y) \in H^{1}_{per}(Y)$ in \eqref{continuityDiff} and using the uniform ellipticity of the coefficients, we have
\[
\begin{split}
\lambda \|\delta_{h}(w^{i})(x, \cdot)\|^{2}_{L^{2}(Y)}&\leq  \int_{Y}(\delta_{h}(A)(x,y)[\nabla_{y}w^{i}(x+h, y) + e^{i}], \nabla \delta_{h}(w^{i})(x,y))dy\\
&\leq \left(\int_{Y}|\delta_{h}(A)(x,y)[\nabla_{y}w^{i}(x+h, y) + e^{i}]|^{2}\right)^{1/2}\|\delta_{h}(w)(x,\cdot)\|_{L^{2}(Y)}
\end{split}
\]
The last inequality implies that
\[
\|\delta_{h}(w^{i})(x,\cdot)\|_{L^{2}(Y)}\leq\Lambda/\lambda \sum_{i=1}^{N} \left(\int_{Y}|\chi^{i}(x+h,y)-\chi^{i}(x,y)|^{2}|\nabla_{y}w^{i}(x+h, y) + e^{i}|^{2}dy\right)^{1/2}
\]
By Meyer's higher regularity result, $\nabla_{y}w^{i}(x+h, \cdot)\in L^{p}(Y)$ for some $p>2$. Moreover, the $L^{p}$ norm is bounded from above by a constant $C$ independent of $x$,  and $h$. After applying Holder's inequality we get
\[
\|\delta_{h}(w^{i})(x,\cdot)\|_{L^{2}(Y)}\leq\frac{C\Lambda}{\lambda} \sum_{i=1}^{N}\left(\int_{Y}|\chi^{i}(x+h,y)-\chi^{i}(x,y)|^{2}dy\right)^{1/2}
\]
Applying \eqref{continuity}, the right hand side approaches $0$ as $h\to 0$ and the proof is complete.

\section*{Acknowledgment}
This work is supported by  grants: NSF DMS-0807265 and AFOSR FA9550-05-0008.  Tadele Mengesha gratefully acknowledges the support of Coastal Carolina University, where he is an assistant professor and currently on leave of absence.

\end{document}